\documentclass[12pt,reqno]{amsart}

\usepackage{graphicx,amsfonts,amssymb,amsmath,amsthm,xpatch}
\usepackage{fullpage}
\usepackage[svgnames]{xcolor}
\usepackage[hyphens,spaces,obeyspaces]{url}
\usepackage{hyperref}
\hypersetup{
  colorlinks = true,
  linkcolor =red,
  anchorcolor = red,
  citecolor = DarkGreen,
  urlcolor = blue
}
\hypersetup{
  pdftitle={Semiclassical measures through Coulomb collisions},
  pdfauthor={Nicholas Lohr},
  pdfsubject={Semiclassical analysis; Coulomb Hamiltonian; Moser regularization},
  pdfkeywords={semiclassical measures, Coulomb Hamiltonian, Moser regularization, hydrogen atom, geodesic flow}
}
\usepackage[all,arc]{xy}
\usepackage{tikz-cd,float}
\usepackage{mathrsfs,mathtools}
\usepackage{enumitem}
\usepackage{aligned-overset}
\usepackage{wrapfig,subcaption}

\newtheorem{theo}{{\sc Theorem}}[section]

\newtheorem{defn}[theo]{{\sc Definition}}

\newtheorem{cor}[theo]{{\sc Corollary}}

\newtheorem{lem}[theo]{{\sc Lemma}}

\theoremstyle{definition}
\newtheorem{example}[theo]{{\it Example}}
\newtheorem{rem}[theo]{{\it Remark}}

\allowdisplaybreaks

\makeatletter
\def\blfootnote{\gdef\@thefnmark{}\@footnotetext}
\def\@cite#1#2{[\textbf{#1}\if@tempswa, #2\fi]}
\newcounter{proofpart}
\xpretocmd{\proof}{\setcounter{proofpart}{0}}{}{}
\newcommand{\proofpart}[2]{%
  \par
  \addvspace{\medskipamount}%
  \stepcounter{proofpart}%
  \noindent\emph{#1 \theproofpart: #2}\par\nobreak\smallskip
  \@afterheading
}
\makeatother

\DeclareFontFamily{U}{mathx}{\hyphenchar\font45}
\DeclareFontShape{U}{mathx}{m}{n}{
      <5> <6> <7> <8> <9> <10>
      <10.95> <12> <14.4> <17.28> <20.74> <24.88>
      mathx10
      }{}
\DeclareSymbolFont{mathx}{U}{mathx}{m}{n}
\DeclareFontSubstitution{U}{mathx}{m}{n}
\DeclareMathAccent{\widecheck}{0}{mathx}{"71}
\DeclareMathAccent{\wideparen}{0}{mathx}{"75}

\title{Semiclassical measures through Coulomb collisions}

\author{Nicholas Lohr}
\address{Purdue University, West Lafayette, IN, 47901}
\email{nlohr@purdue.edu}
\date{\today}

\begin{document}
\blfootnote{The author has no competing interests to declare that are relevant to the content of this article.}
\begin{abstract}
We prove that $\mu$ is a semiclassical measure associated to a
sequence of eigenfunctions of energy $E<0$ of the attractive
Coulomb operator if and only if $\mu$ is a probability measure
on the energy (hyper)surface $\Sigma_E$ invariant under the
regularized Kepler flow due to Moser. The converse was shown in
recent work by the author \cite{L23}, and the present article
proves the other direction (as well as an independent proof of
the converse). We prove the main theorem for a general symbol
class allowing certain non-decay at infinity, which implies that
semiclassical measure mass entirely reflects off of the
origin. In the special case of semiclassical measures of sequences of eigenfunctions of the exact Coulomb operator, this
article solves an open problem posed by Keraani in \cite[Remark 1.11]{K05}. The main tools include the celebrated Moser-Fock map along with a technical operator extension lemma in $\Psi_{\hbar}^0(\mathbb{S}^d)$, which utilizes standard eigenfunction concentration bounds. 
\end{abstract}
\maketitle
\section{Introduction}
In this article, we further study the semiclassical
measures corresponding to eigenfunctions of the attractive Coulomb operator, defined as
\begin{equation}\label{eq:operator}
\widehat{H}_{\hbar}:L^2(\mathbb{R}^d) \to L^2(\mathbb{R}^d),\quad \widehat{H}_{\hbar}\coloneqq -\frac{\hbar^2}{2}\Delta-\frac{1}{|x|},\quad  \hbar>0,\quad d\geq 3.
\end{equation}
When $d=3,$ this operator is the first approximation of the quantum
hydrogen atom. That is, fixing the reduced mass of the
electron-proton system to $1$, the reduced Bohr radius to
$\hbar^2$, and ignoring all relativistic and spin-coupling
effects, the Schr\"odinger operator for the relative position of
the electron is given by $\widehat{H}_{\hbar}$. The operator is initially defined on $C_c^{\infty}(\mathbb{R}^d)$ by Hardy's inequality and extended to a self-adjoint Friedrichs extension on $L^2(\mathbb{R}^d)$
with domain $H^2(\mathbb{R}^d)$. We recall this and other well-known spectral theory of $\widehat{H}_{\hbar}$ in \S\ref{subsec:spec}. 
\par The attractive Coulomb operator corresponds to the
classical phase space Hamiltonian
\begin{equation*}
H:T^*(\mathbb{R}^d\setminus\{0\}) \to \mathbb{R},\quad H(x,\xi)\coloneqq \frac{|\xi|^2}{2}-\frac{1}{|x|},
\end{equation*}
called the Kepler Hamiltonian, where we identify
$T^*(\mathbb{R}^d\setminus\{0\})=\mathbb{R}^d\setminus\{0\}
\times \mathbb{R}^d$ using the Riemannian metric on
$\mathbb{R}^d\setminus\{0\}$. For a fixed energy $E$, the
Hamiltonian orbits, also called Kepler orbits, lie on the energy
(hyper)surface
\begin{equation*}
\Sigma_E \coloneqq \{(x,\xi) \in T^*(\mathbb{R}^d \setminus \{0\}) \mid H(x,\xi)=E\}.
\end{equation*}
For any energy $E \in \mathbb{R}$, $\Sigma_E$ is \textit{not}
compact due to the $x \to 0, \xi \to \infty$ regime, see Figure \ref{fig:4}.

  \begin{figure}[!htbp]
\centering
\begin{subfigure}{.65\textwidth}
\centering
  \includegraphics[width=0.75\linewidth]{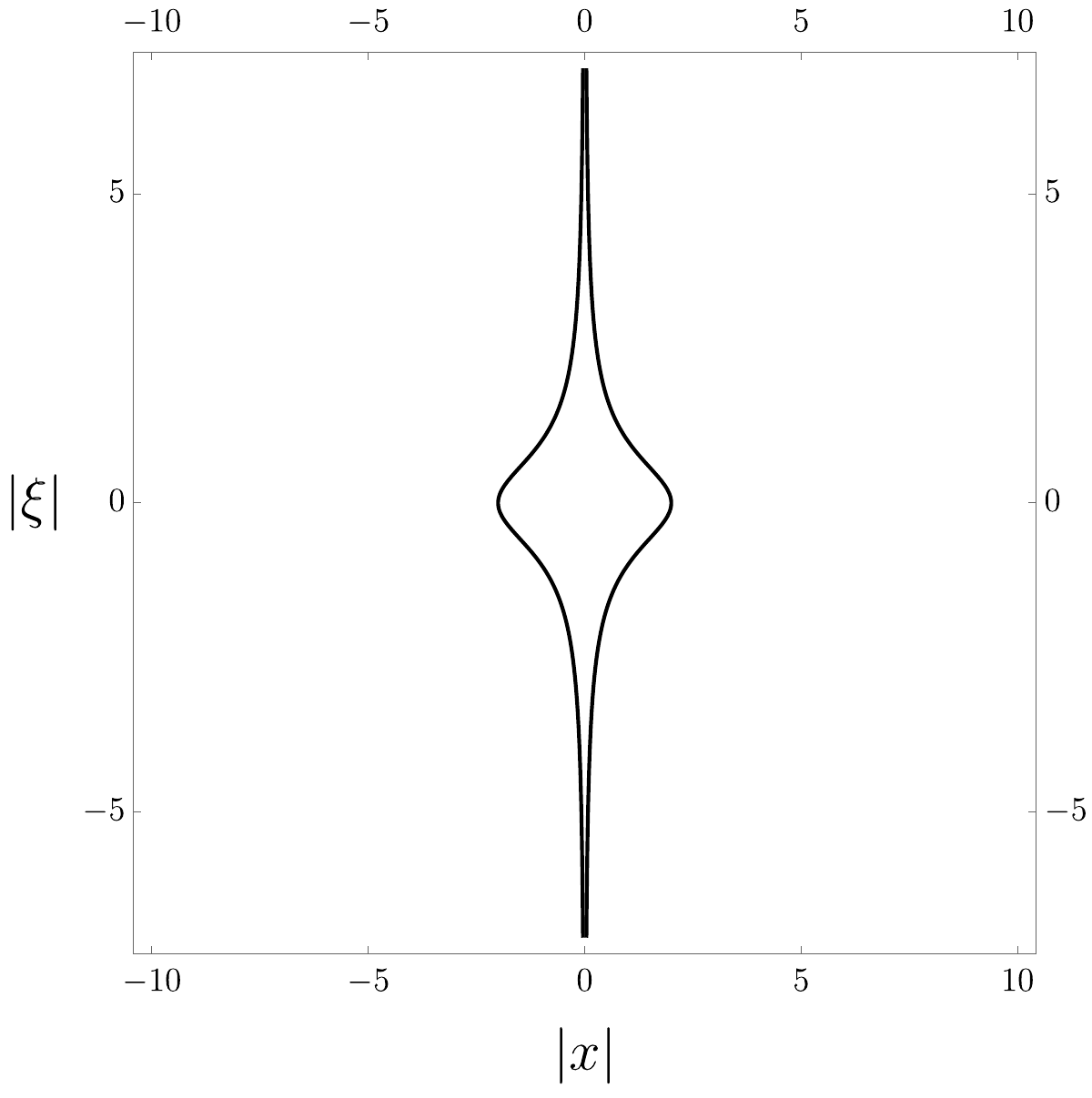}%
  \caption{Plot of $H(x,\xi)=-\frac{1}{2}$.}\label{fig:4a}
  \end{subfigure}%
  \hfill
\begin{subfigure}{.25\textwidth}
  \centering
  \includegraphics[width=.5\linewidth]{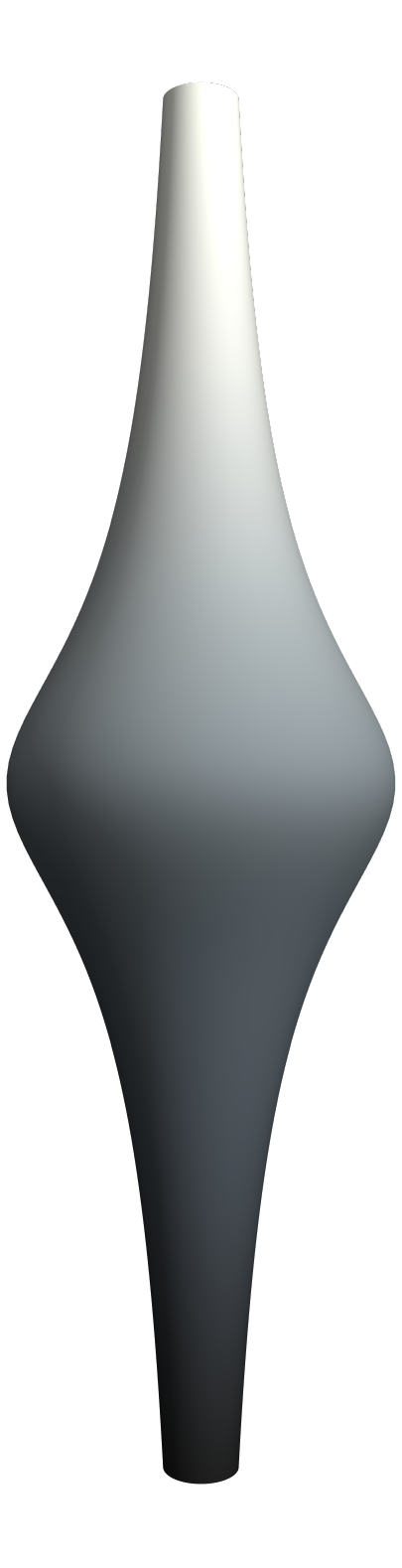}%
  \caption{Higher-dim\-en\-sion\-al schematic of $H=-\frac{1}{2}$\sloppy}\label{fig:4b}
  \end{subfigure}%
  \caption{Low-dimensional plots of $\Sigma_E$ when $E=-1/2$}\label{fig:4}%
  \end{figure}

For $E<0$, the orbits consist of two types: periodic orbits
whose configuration space projections are planar ellipses, and
unbounded ``collision'' orbits whose configuration space projections
are line segments terminating at the origin in finite
time, see Figure \ref{fig:1a}. The configuration space projections of the periodic
Kepler orbits follow Kepler's laws of planetary motion (with one
body fixed and all physical constants fixed to $1$). Namely, the periodic configuration space trajectories
\begin{itemize}
\item are ellipses with the origin fixed at one focus,
\item are such that the line segment connecting the trajectory
to the origin sweeps out equal areas during equal time
intervals,
\item have period $T$ related to the energy $E$ by the
formula
\begin{equation}\label{eq:kepler3}
T=\frac{2\pi}{p_0^3},\quad p_0 \coloneqq \sqrt{-2E},
\end{equation}
\end{itemize}
where we have used our convention on physical constants. Observe
that Kepler’s third law is popularly stated with the length of
the semi-major axis $a$, but, with our conventions, $a=p_0^{-2}$
(see \cite[(5)]{M83} and \cite[\S1.1.2.]{L25}).
\par This Hamiltonian system is not only completely integrable, but it is maximally superintegrable with $2d-1$ independent integrals of motion among the components of the
conserved quantities of the Hamiltonian $H$, the angular momentum $(d-2)$-vector $L$, and
the Runge-Lenz eccentricity vector $R$ defined by
\begin{equation}\label{eq:rungelenz}
L(x,\xi) \coloneqq \star(x \wedge \xi),\qquad R(x,\xi)=\Big(|\xi|^2-\frac{1}{|x|} \Big)x-(x\cdot \xi)\xi,
\end{equation}
where $\star$ denotes the Hodge-star operator.\par
On $\Sigma_E$, the magnitudes of these quantities are related by
the formula
$$|R|^2=1+2E|L|^2. $$
A Kepler orbit is a collision orbit if and only if $L=0$. Provided that $L\neq 0$, in configuration space, $L$ determines the plane of motion, $|R|$ is the eccentricity of the ellipse, $R$ and the foci are collinear, and $|2E|^{-1}$
is the length
of the semi-major axis (as noted previously). The Runge-Lenz vector $R$ has a long, complicated history of discovery and rediscovery (see the works of Goldstein \cite{G75,G76}), but,
most noteworthy, Hamilton in \cite{H47} showed that the Runge-Lenz vector can be understood as coming from the geometry of the momentum space projections of the Kepler orbits,
which miraculously happen to be circles. Each circle has radius $1/|L|$ and is centered at the point obtained by rotating $R/|L|$ by $90^{\circ}$ in the plane of motion (more
carefully, these circles degenerate into lines for the collision orbits). The superintegrability
explains why the bounded orbits are not merely quasi-periodic and confined to invariant tori
as guaranteed from the Liouville-Arnold theorem (see \cite[Chapter 10]{A89}), but the bounded
orbits are genuinely \textit{periodic} (see \cite{GS90} for more on the symmetries of this problem).
\par However, because of the collision orbits, the Hamiltonian
flow of $H$ is \textit{not} complete. In \cite{M70}, Moser
compactified $\Sigma_E$ to a manifold $\overline{\Sigma_E}$
(defined in \eqref{eq:compact}) where the Hamiltonian flow is
regularized by a reflection condition. Roughly speaking, when
the collision orbits hit the origin, they are reflected back
along the same line, resembling a degenerate ellipse. The
manifold $\overline{\Sigma_E}$ is diffeomorphic to $S^*
\mathbb{S}^d$, and, up to a reparametrization, the regularized
Hamiltonian flow maps to the cogeodesic flow on
$\mathbb{S}^d$. In particular, the collision orbits are mapped
to the great circles passing through the `north pole' of
$\mathbb{S}^d$. This completes the Hamiltonian flow and extends
the collision orbits past their finite collision time to be
periodic on all time and obeying Kepler's third law
\eqref{eq:kepler3}.
\par For fixed $E<0$ and sequences $\hbar_j \to 0$, $E_{\hbar_j} \to E$, we say that a sequence $\Psi_j$
of $L^2$-normalized eigenfunctions of $\widehat{H}_{\hbar_j}$
satisfying
\begin{equation}\label{eq:EIGGG}
\widehat{H}_{\hbar_j}\Psi_j=E_{\hbar_j} \Psi_j
\end{equation}
converges to a nonnegative Radon measure $\mu$ on $T^*\mathbb{R}^d$ in the sense of semiclassical
measures if, for any $a \in C_c^{\infty}(T^*\mathbb{R}^d)$, we
have
\begin{equation*}
\langle \operatorname{Op}_{\hbar_j}(a)\Psi_{j},\Psi_{j}
\rangle \xrightarrow{j \to \infty }\int_{T^*\mathbb{R}^d}a(x,\xi)
d\mu(x,\xi), 
\end{equation*}
where $\operatorname{Op}_{\hbar}$ denotes semiclassical Weyl
quantization. For the main theorem statement, we say that a nonnegative Radon
measure $\mu$ on $\Sigma_E$ is invariant under the Hamiltonian
flow if the pushforward measure $(i_{\Sigma_E})_*\mu$ is invariant under the Moser-regularized Hamiltonian flow $\overline{\Xi_H^{\bullet}}$ (defined in \eqref{eq:reghamil}) where $i_{\Sigma_E}:\Sigma_E\to \overline{\Sigma_E}$ is the inclusion map (defined in \eqref{eq:inclusion}).
\subsection{Statement of Results}
\begin{theo}\label{theo:1}
With test functions $a \in C_c^{\infty}(T^*\mathbb{R}^d)$, $\mu$ is a semiclassical measure associated to a sequence of eigenfunctions of energy $E<0$ of $\widehat{H}_{\hbar}$ if and only if $\mu$ is a probability measure supported on $\Sigma_E$ invariant under the Moser-regularized Hamiltonian flow.
\end{theo}
The `if' statement was the main result of \cite{L23}, and this
paper proves the `only if' statement (although we give another
proof of the `if' statement in the present article). \par In
smooth settings, proving the `only if' statement, that semiclassical measures of
eigenfunctions are probability measures supported on the energy
hypersurface and invariant under the Hamiltonian flow, is relatively
straightforward, as seen in the following example.
\begin{example}\label{ex:main}

Suppose $(M,g)$ is a compact smooth
manifold and $\{u_h\}_{h>0}$ is a sequence of $L^2$-normalized
$1$-eigenfunctions of $-h^2\Delta_{M}$. Then any semiclassical
measure $\mu$ of $\{u_h\}_{h>0}$ satisfies
\begin{gather}
\operatorname{supp} \mu \subseteq S^*M,\label{eq:supp}\\
(\Xi_g^t)_*\mu=\mu, \label{eq:inv}\\
  \mu(S^*M)=1,\label{eq:probb}
\end{gather}
where $\Xi_g^t$ denotes the (co)geodesic flow. Properties \eqref{eq:supp},\eqref{eq:inv} follow from considering the test functions $a(u,\eta)(|\eta|_g^2-1),\{a(u,\eta),|\eta|_g^2-1\},$ respectively, where $a \in C_c^{\infty}(T^*M)$ and $\{\bullet,\bullet\}$ is the Poisson bracket (see \cite[Theorems E.34, E.35]{DZ19} for more details). Property \eqref{eq:probb} follows from the fact that the semiclassical convergence statement involving test functions in $C_c^{\infty}(T^*M)$ remains true in the Kohn-Nirenberg symbol class $S_{1,0}^0(T^*M)$, which is a consequence of $M$ being compact and $-h^2\Delta_M$ being semiclassically elliptic at fiber infinity (see \cite[Exercise E.23]{DZ19}).
\end{example}
Unfortunately, the arguments used to prove the statements in Example \ref{ex:main}
cannot be applied fully to $\widehat{H}_{\hbar}$, as the
singular $1/|x|$ potential causes serious issues. While an
argument analogous to the one used to prove \eqref{eq:supp} can
prove $\{x \neq 0\} \cap \operatorname{supp}\mu \subseteq
\Sigma_E$, it is nontrivial to investigate the region where $x$
is allowed to vanish. Additionally, since $\Sigma_E$ is not compact, mass
could `leak' into the origin as $x \to 0,\xi \to \infty,$ which
would lead to $\mu(\Sigma_E)<1$. For statements about
Hamiltonian flow invariance, we recall that the classical
Hamiltonian flow on $\Sigma_E$ is \textit{not} complete, so we
have to carefully define what it even means for a measure on
$\Sigma_E$ to be invariant under the Hamiltonian flow. Because
of this, any argument similar to the usual one proving
\eqref{eq:inv} does not directly apply in this context.
\par Theorem \ref{theo:1} shows that any leakage of mass in the
origin is not possible. In some sense, the semiclassical
measures only `see' Moser's regularization when it comes to the
collision orbits. More carefully, we have the following
corollary:
\begin{cor}
Let $\gamma$ be a Kepler phase space orbit of energy $E<0$ with initial point $\gamma(0) \in \Sigma_E$. There exists a semiclassical measure $\mu$ of Coulomb
eigenfunctions of energy $E<0$ satisfying
$\mu=\delta_{\gamma}$ if and only if $\gamma$ is not a
collision orbit, where $\delta_{\gamma}$ is defined by
$$\int a d \delta_{\gamma} \coloneqq
\frac{1}{t_*}\int_0^{t_*}a(\gamma(t))dt, $$
where $t_*$ denotes either the period time $T=2\pi/p_0^3$ of $\gamma$ (in the case of a non-collision orbit, see \eqref{eq:kepler3}) or the collision time (in the case of a collision orbit, see Remark \ref{rem:collision}). 
\end{cor}
However, if $\overline{\gamma}$ denotes a Moser-regularized
collision orbit in phase space, then Theorem \ref{theo:1}
shows that there does indeed exist a semiclassical measure $\mu$
of Coulomb eigenfunctions of energy $E<0$ satisfying
$\mu=\delta_{\overline{\gamma}}$.
\par A natural question is whether Theorem \ref{theo:1} remains
true for a more general class of symbols that have support when
$x\to 0, \xi \to \infty$. Again, due to the non-compactness of
$\Sigma_E$, one has to be careful defining semiclassical
measures in this context. A natural candidate for symbols would
be those that work well with semiclassical analysis, and, when they are
restricted to $\Sigma_E$, they extend smoothly to
$\overline{\Sigma_E}$. Indeed, we define
\begin{equation}\label{eq:symb1}
S_{\overline{\Sigma_E}} \coloneqq \{a \in S(1)\, : \,a|_{\Sigma_E} \text{ extends smoothly to }\overline{\Sigma_E}\}
\end{equation}
where we refer to Remark \ref{rem:bars} for the definition of smoothness on $\overline{\Sigma_E}$, and $S(1)= S_{0,0}^0(T^*\mathbb{R}^d)$ denotes the
Kohn-Nirenberg class
\begin{equation}\label{eq:kn}
S_{0,0}^0(T^*\mathbb{R}^d) \coloneqq \{ a \in C^{\infty}(T^*\mathbb{R}^d)\, : \,  |\partial^{\alpha}a|\leq C_{\alpha}\text{ for all }\alpha \in \mathbb{N}^{2d}\}.
\end{equation}
For fixed $E<0$, we then say that a measure $\mu$ on $\overline{\Sigma_E}$ is a
semiclassical measure of a sequence of $L^2$-normalized $E_{\hbar_j}$-eigenfunctions $\Psi_j$ of $\widehat{H}_{\hbar_j}$ where $\hbar_j \to 0$, $E_{\hbar_j} \to E$
if, for any $a \in S_{\overline{\Sigma_E}}$, we
have
\begin{equation}\label{eq:seminew}
\langle \operatorname{Op}_{\hbar_j}(a)\Psi_{j},\Psi_{j}
\rangle \xrightarrow{j \to \infty }\int_{\overline{\Sigma_E}}\overline{a} d\mu,
\end{equation}
where $\operatorname{Op}_{\hbar}$ denotes semiclassical Weyl
quantization.
\begin{theo}\label{theo:2}
Theorem \ref{theo:1} remains true for test functions $a \in S_{\overline{\Sigma_E}}$. That is, for test functions $a \in S_{\overline{\Sigma_E}}$, $\mu$ is a semiclassical measure on $\overline{\Sigma_E}$ associated to a sequence of eigenfunctions of energy $E<0$ of $\widehat{H}_{\hbar}$ (as in \eqref{eq:seminew}) if and only if $\mu$ is a probability measure on $\overline{\Sigma_E}$ invariant under the Moser-regularized Hamiltonian flow.
\end{theo}
\begin{rem}
Although the class $S_{\overline{\Sigma_E}}$ is rather
technical, it does contain many familiar symbols in $S(1)$. For
example, using the condition \eqref{eq:ainv}, one can check that
it contains symbols that are constant near $\xi \to \infty, x
\to 0$, as well as certain symbols of one variable, like the
Schwartz spaces
$\mathcal{S}(\mathbb{R}_x^d),\mathcal{S}(\mathbb{R}_{\xi}^d)$.
\end{rem}
The strategy of the proofs of these theorems is to directly use
the Moser-Fock correspondence in the quantum expectations, thereby essentially reducing the above theorems to foundational results of semiclassical measures of eigenfunctions on the sphere \cite[Theorem 1.1]{JZ99}. We prove that in the semiclassical limit, the peculiar map $\langle \hbar D/\sqrt{-2E_{\hbar}} \rangle$ of the Fock map \eqref{eq:FOCK} accounts for the classical reparametrization \eqref{eq:timechange}. The remaining maps of \eqref{eq:FOCK} work well with the symmetries of Weyl quantization and correspond to the maps of \eqref{eq:MOSER} in the natural way. The map $\mathsf{S}$ in \eqref{eq:MOSER} corresponds to scaling the semiclassical parameter $\hbar$ to $\hbar/p_0$. In the proof, a technical issue arises extending an operator natural on $\mathbb{S}_{\neq \mathsf{NP}}^d$ to an operator $\mathbb{S}^d$, which is handled by Lemma \ref{lem:ext2} (see the weaker Lemma \ref{lem:ext1} as well).
\par In general, it is hard to characterize the set of all
semiclassical measures $\mu$ for a given operator.  In the non-chaotic setting, the set of semiclassical measures has been completely characterized in a handful of smooth settings; see the introductions of \cite{A20,AM22} for accounts of the literature. In the chaotic setting, the set of semiclassical measures is \textit{almost} characterized by the quantum ergodicity theorem. One instance of this theorem is the following: if $M$ is
a compact, smooth Riemannian manifold without boundary such that
cogeodesic flow is ergodic with respect to the Liouville
measure, then any orthonormal sequence of $1$-eigenfunctions of the
semiclassical Laplace-Beltrami operator of $M$ admits a density-one subsequence that converges to the
Liouville measure in the sense of semiclassical measures as
$\hbar \to 0$ (see \cite{S74a,S74b,L93,Z87,CdV85} for the
original works and \cite{D22} for an exposition of the results
in the chaotic setting). The quantum unique ergodicity
conjecture states that it is not necessary to descend to a
density-one subsequence and thus completely characterizes the set
of semiclassical measures in this setting, but this conjecture
is still open.
\par There are relatively few results on semiclassical measures
in singular settings. Recently, Galkowski and Wunsch proved that the standard propagation theorem of semiclassical measures of Schr\"odinger operators continues to hold in cases when the potential can be rather rough \cite{GW24}. To the author's knowledge, the first result characterizing the semiclassical measures in a singular setting is the recent work \cite{L23}. Other results include uniform non-concentration estimates of semiclassical measures of one-dimensional Schr\"o\-ding\-er operators found in \cite{HM23}, and very recent work of \cite{V26} studying semiclassical measures of Dirac-delta perturbations of $-h\Delta_{\mathbb{S}^2}$.
\par The implications of the regularized Hamiltonian flow on the quantum dynamics of Schr\"o\-ding\-er operators with Coulomb-like potentials have also been well-studied. G\'{e}rard and Knauf in \cite{GK91} showed that the semiclassical wavefront set of solutions $u_{\hbar}(t)=e^{-it \widehat{H}_{\hbar}/\hbar}u_{\hbar,0}, u_{\hbar,0} \in L^2(\mathbb{R}^3)$ of the time-dependent Schr\"{o}dinger equation propagates along regularized Hamiltonian orbits, including beyond the collision time. Additionally, Keraani in \cite{K05} showed the analogous statement for the propagation of semiclassical measures initially supported away from the origin. These papers regularize the Hamiltonian flow through the Kustaanheimo-Stiefel (KS) transformation. The KS map reduces this three-dimensional Hamiltonian flow to a suitably constrained four-dimensional harmonic oscillator flow (see the original works of \cite{K64,KS65} as well as the book \cite{SS71}), and it is the three-dimensional generalization of the one-dimensional and two-dimensional regularizations of the Kepler problem known to Euler \cite{E74} and Levi-Civita \cite{LC20}, respectively. Although the KS transformation has proven to be a powerful tool as exhibited in the aforementioned \cite{GK91,K05} and other work such as \cite{CJK08}, it has several drawbacks. The inverse KS map is only locally defined by introducing a dummy variable defined on the circle, the KS map also has no obvious generalization to dimensions higher than three, and, to the author's knowledge, it has no obvious `quantization' that relates the spectrum of the four-dimensional harmonic oscillator to that of the Coulomb operator. We note that the unitary Fock map has a satisfactory answer to these three defects, and we use these additional properties in this article.
\par The point of this article is to complete the work of
\cite{L23} and prove that \textit{all} semiclassical measures
are probability measures on the energy hypersurface invariant under
the Moser-regularized flow. Our proof generalizes to more
general symbols allowing mass near $x\to0,\xi \to \infty$. In
the case of sequences of eigenfunctions of the exact Coulomb
operator, we extend the results of \cite{K05} and
prove the open problem remarked in \cite[Remark
1.11]{K05}. In the case of semiclassical measures $\mu$ of eigenfunctions, Keraani proves
\begin{equation}\label{eq:keraani}
\mu|_{x \neq
  0}=(\overline{\Xi_{H}^t})_*(\mu|_{T^*\mathbb{R}_{\neq
    0}^3\setminus \operatorname{Coll}_t}), 
\end{equation}
where $\overline{\Xi_H^{\bullet}}$ is the Moser-regularized flow defined in \eqref{eq:reghamil}, and $\operatorname{Coll}_t$ is the set of initial points whose Moser-regularized flow hits $x=0$ at time $t$ (the KS flow and Moser flow coincide for exact Coulomb in $d=3$). The main theorem of this paper removes the exclusion of the collision points and analyzes the flow near $x=0$, all while working in dimension $d \geq 3$.

\subsection{Acknowledgments}
The author would like to thank Antoine Prouff, Andr\'as Vasy, and Jared Wunsch for helpful conversations. 

 \subsection{Background: The Moser Map}
We define the classical Moser map, first defined by Moser in \cite{M70} for $d=3$ (see \cite[\S1.2.2.]{L25} for proofs in general dimensions). This map regularizes the incomplete Kepler flow by mapping the (regularized) Hamiltonian orbits on a compactified $\Sigma_E$ to the geodesics of $S^*\mathbb{S}^d$.
  We use the notation
 $$\mathbb{S}_{\neq \mathsf{NP}}^d\coloneqq \mathbb{S}^d\setminus \{\mathsf{NP}\},\quad \mathsf{NP} \coloneqq (0,0,\ldots,1),$$
 to denote the sphere punctured at the `north pole.'
Let $\omega: \mathbb{R}^d \to \mathbb{S}_{\neq \mathsf{NP}}^d$
be inverse of stereographic projection from the north pole. That
is, the maps $\omega: \mathbb{R}^d \to \mathbb{S}_{\neq \mathsf{NP}}^d$ and $\omega^{-1}:\mathbb{S}_{\neq \mathsf{NP}}^d\to \mathbb{R}^d$ are given by
\begin{equation}\label{eq:stereo}
\omega(x) \coloneqq \frac{1}{|x|^2+1}\begin{cases}
2x_k & \text{if }k<d+1\\

|x|^2-1 &\text{if }k=d+1
\end{cases}, \quad \omega^{-1}(u)_j=\frac{u_j}{1-u_{d+1}},\ j=1,\ldots,d.
\end{equation}
It can be computed that the pullback $\omega^*:T^*\mathbb{R}^d \to T^*(\mathbb{S}_{\neq \mathsf{NP}}^d)$ is
\begin{align}
\omega^*(x,\xi)=(\omega(x),\eta) \quad \text{with} \quad \eta_j =\begin{cases}
           \xi_j \frac{|x|^2+1}{2}-(x\cdot \xi) x_j &\text{ if }
                                               j<d+1\\
           x \cdot \xi &\text{ if }
                                               j=d+1
                                                                  \end{cases}\label{eq:mos},
\end{align}
where we have identified $T^* \mathbb{R}^d\cong T
\mathbb{R}^d=\mathbb{R}_x^d \times \mathbb{R}_{\xi}^d$ and
$T^*(\mathbb{S}_{\neq \mathsf{NP}}^d)\cong T(\mathbb{S}_{\neq
  \mathsf{NP}}^d)\subset T \mathbb{R}^{d+1}=\mathbb{R}_u^{d+1}
\times \mathbb{R}_{\eta}^{d+1}$ with the musical isomorphisms
induced by the respective Riemannian metrics.

\begin{theo}[{{\cite[Theorem 1]{M70}}}]\label{theo:moser} Let $E<0$ and define $p_0 \coloneqq
          \sqrt{-2E}$. Define the Moser map
          \begin{equation}\label{eq:MOSER}
\mathcal{M}_E:T^*\mathbb{R}^d \to T^*(\mathbb{S}_{\neq \mathsf{NP}}^d),\quad \mathcal{M}_E \coloneqq   \omega^* \circ R_{-\pi/2}
          \circ \mathcal{D}_{p_0}\circ \mathsf{S},\quad \begin{array}{c} \mathcal{D}_{p_0}(x,\xi) \coloneqq (p_0 x
                                                           ,p_0^{-1}\xi),\\ R_{-\pi/2}(x,\xi) \coloneqq (\xi,-x),\\ \mathsf{S}(x,\xi)
          \coloneqq (p_0x,\xi),
                                                           \end{array}
\end{equation}
 where $\mathcal{D}_{p_0}$ is the symplectic dilation by $p_0$,
          $R_{-\pi/2}$ is the
          symplectic rotation by $-\pi/2$, $\mathsf{S}$ is a nonsymplectic dilation, and $\omega^*$ is the pullback of $\omega$, found in \eqref{eq:mos}.
Up to a reparametrization of time, the Moser map $\mathcal{M}_E$ transforms the Kepler flow on $\Sigma_E$ onto the cogeodesic flow on $S^*(\mathbb{S}_{\neq \mathsf{NP}}^d)$ parametrized by arc length. \par More specifically, 
if $\gamma(t)=(x(t),\xi(t)) \in T^*\mathbb{R}^d$ is a Kepler orbit on $\Sigma_E$, then $\varphi(s)=(u(s),\eta(s)) \coloneqq \mathcal{M}_E(\gamma(t(s)) \in S^*(\mathbb{S}_{\neq \mathsf{NP}}^d)$ is a 
cogeodesic on $S^*(\mathbb{S}_{\neq \mathsf{NP}}^d)$ parametrized by arc length $s$ where
$t(s)$ satisfies
\begin{equation}\label{eq:timechange}
\frac{dt}{ds}=\frac{|x(t(s))|}{p_0}=\frac{1-u_{d+1}(s)}{p_0^3},\quad t(0)=0.
\end{equation}
\end{theo}

\begin{figure}[!htbp]
\centering
\begin{subfigure}{.5\textwidth}
  \centering
  \includegraphics[width=.7\linewidth]{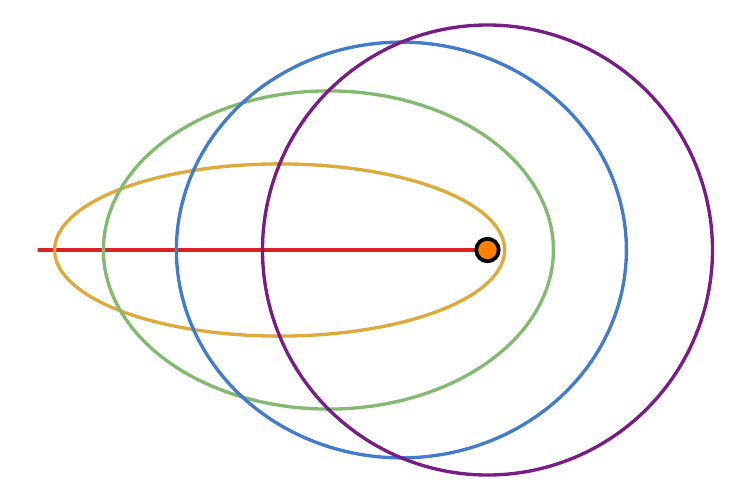}%
  \caption{Position graphs}\label{fig:1a}
\end{subfigure}%
\begin{subfigure}{.5\textwidth}
  \centering
  \includegraphics[width=.6\linewidth]{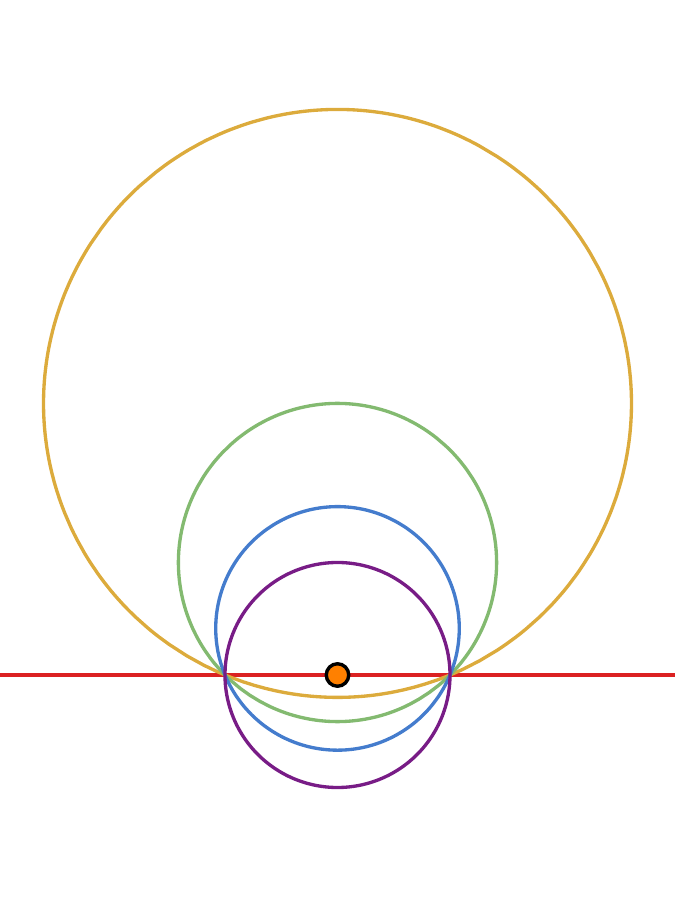}%
  \caption{Momentum graphs}\label{fig:1b}
  \end{subfigure}%
  \caption{Position and momentum graphs when $E=-1/2$ \& $0 \leq |L| \leq 1 $}\label{fig:1}
  \end{figure}
  \begin{figure}[!htbp]
\centering
  \includegraphics[width=.33\linewidth]{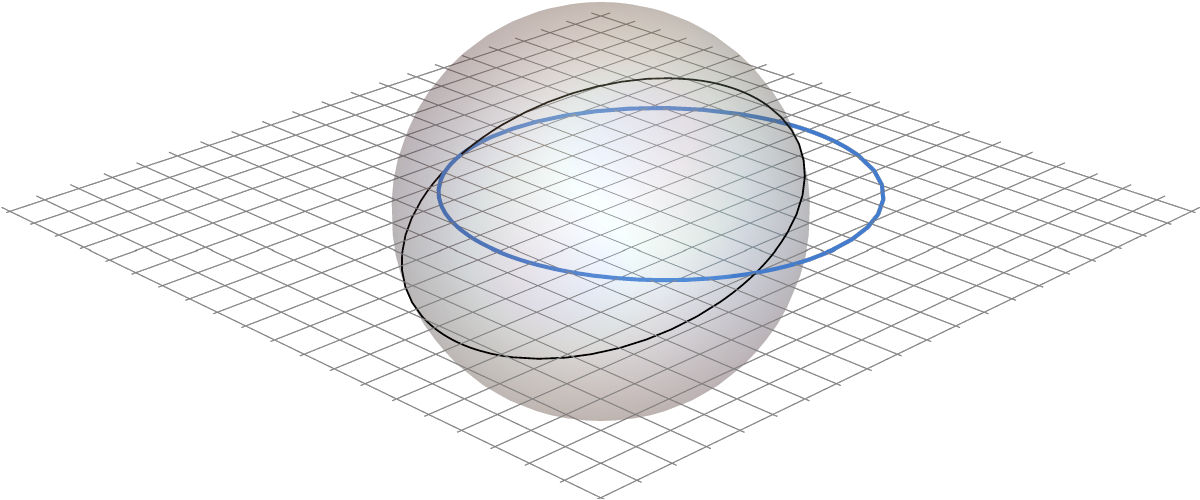}%
  \includegraphics[width=.33\linewidth]{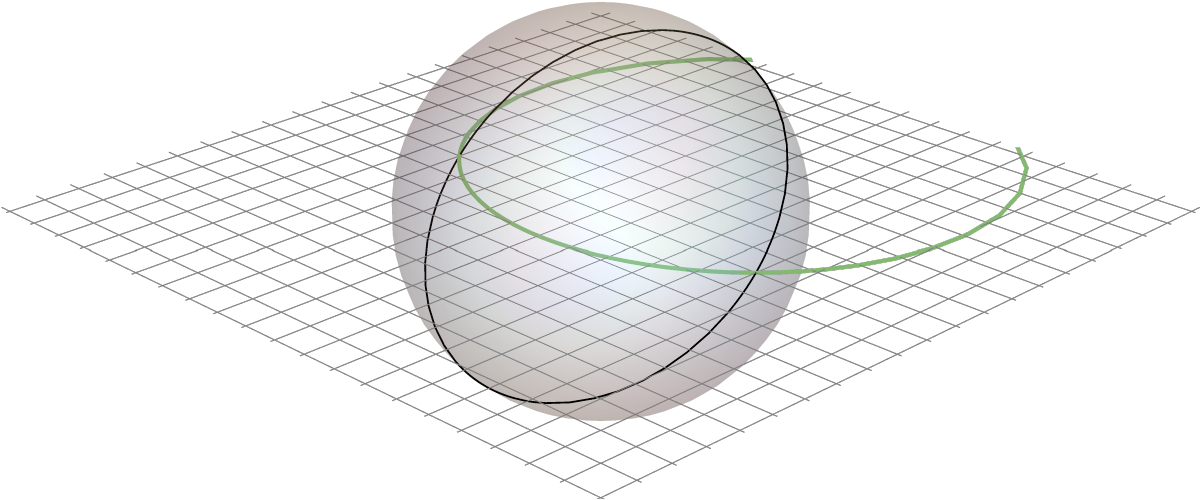}%
  \includegraphics[width=.33\linewidth]{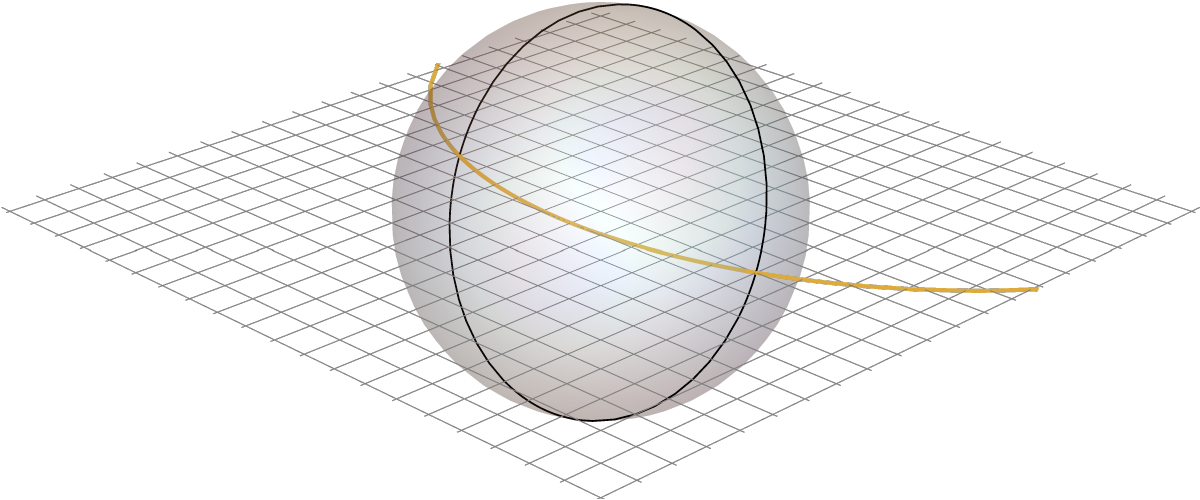}%
\caption{Snapshots of momentum graphs as $|L|$ varies from $1$ to $0$}\label{fig:2}
  \end{figure}

\begin{rem} The initial scaling maps $\mathsf{S}$ and $\mathcal{D}_{p_0}$
are compatible with the inhomogeneity of $H$ as follows:
$$H\big((\mathsf{S} \circ
\mathcal{D}_{p_0})(x,\xi)\big)=H(p_0^2x,p_0^{-1}\xi)=p_0^{-2}H(x,\xi). $$
On the first reading, one may set $p_0=1$ and thus make $\mathsf{S}=\mathcal{D}_{p_0}=I$. We caution that this map (and its quantization) \textit{crucially} depends on the energy, so this simplification is rather dangerous. The appearance of $\omega^*$ and $R_{-\pi/2}$ can be seen visually: compare the momentum graphs in Figure \ref{fig:1b} to the stereographic projection of great circles in Figure \ref{fig:2}.
\end{rem}
\begin{rem}[Symmetries]
One can compute
\begin{equation}\label{eq:sympl}
\mathcal{M}_E^*(\eta\cdot du)=-p_0x\cdot d\xi
\end{equation}
where $\eta\cdot du$ denotes the tautological 1-form on $T^*\mathbb{R}^{d+1}$ restricted to $T^*(\mathbb{S}_{\neq \mathsf{NP}}^d)$, which implies $\mathcal{M}_E$ is a (conformal) symplectomorphism. Additionally, the functions $u_j\eta_k-u_k\eta_j$ on $T^*(\mathbb{S}_{\neq \mathsf{NP}}^d)$ pulled back by $\mathcal{M}_E$ can be computed as
\begin{align}
\mathcal{M}_E^*(u_j\eta_k-u_k\eta_j)&=p_0(x_j\xi_k-x_k\xi_j),\quad
                          j,k \neq d+1, \label{eq:preang}\\
  \mathcal{M}_E^*(u_j\eta_{d+1}-u_{d+1}\eta_j)&=\frac{|\xi|^2-p_0^2}{2}x_j-(x \cdot \xi)\xi_j,\quad
                          j \neq d+1 \label{eq:prelenz}.
\end{align}
That is, \eqref{eq:preang} states that $\mathcal{M}_E$ pulls back the components of angular
momentum not involving the last coordinate in $\mathbb{R}^{d+1}$
to \textit{all} the (scaled) components of angular momentum in
$\mathbb{R}^d$. Put differently, for $g \in
\operatorname{SO}(d)$, we have
\begin{equation}\label{eq:later}
\mathcal{M}_E \circ g^*=\begin{pmatrix} g & 0\\ 0 & 1\end{pmatrix}^* \circ \mathcal{M}_E,
\end{equation}
where the asterisk denotes the symplectic lift of the rotation action on the base manifold to the cotangent bundle. 
\par To further understand \eqref{eq:prelenz}, we
first observe that one can check
$\mathcal{M}_E|_{\Sigma_E}=S^*(\mathbb{S}_{\neq
  \mathsf{NP}}^d)$. On $\Sigma_E$, the right hand side of
\eqref{eq:prelenz} coincides with the components of the Runge-Lenz vector $R$ in
\eqref{eq:rungelenz}.\par
We conclude this remark by noting that \eqref{eq:sympl} implies
\begin{equation}\label{eq:sympl2}
\mathcal{M}_E^*\{f,g\}=p_0^{-1} \{f \circ \mathcal{M}_E,g \circ \mathcal{M}_E\}
\end{equation}
for any $f,g \in C^{\infty}(T^*\mathbb{R}^{d+1})$. This shows that up to a scaling, the components of angular momentum and $\widetilde{R}_E \coloneqq \frac{|\xi|^2-p_0^2}{2}x-(x \cdot \xi)\xi$ in $\mathbb{R}^d$ share the same commutation relations as all of the components of angular momentum in $\mathbb{R}^{d+1}$. Thus $\operatorname{span}_{\mathbb{R}}(L_{j,k},(\widetilde{R}_E)_j\mid j,k)$ together with the Poisson bracket forms an $\mathfrak{so}(d+1)$ Lie algebra, showing the hidden symmetry.
\end{rem}

\begin{rem}[Time Change] We note that
\eqref{eq:timechange} includes the identity
\begin{equation}\label{eq:timechange3}
\mathcal{M}_E^*(1-u_{d+1})=p_0^2|x|,
\end{equation}
which we use in the proof of the main theorem. Note that $t(s)$ can be computed explicitly. Indeed,
the geodesic on $\mathbb{S}^d$ generated by $(u_0,\eta_0) \in
S^*\mathbb{S}^d$ is
$$u(s)=u_0\cos s +\eta_0 \sin s, $$
so integrating \eqref{eq:timechange} gives
\begin{equation}\label{eq:timechange2}
t(s)=\frac{1}{p_0^3}\big(s-(u_0)_{d+1}\sin s -(\eta_0)_{d+1}\cos s+(\eta_0)_{d+1}\big).
\end{equation}
Now $t(s)$ is increasing since $t'(s) \geq 0$ and $t'(s)=0$ only at the discrete, periodic points $s$ where $u_{d+1}(s)=1$. Substituting $s=2\pi$ into \eqref{eq:timechange2} recovers Kepler's third law \eqref{eq:kepler3}.
\end{rem}
\begin{rem}[Collision Orbits]\label{rem:collision}
In the case of a collision orbit (that is, when $L=0$), by
\eqref{eq:preang}, we see that the corresponding geodesic on
$\mathbb{S}_{\neq \mathsf{NP}}^d$ has zero angular momentum in
the directions not involving the last coordinate. That is, the
collision Kepler orbits correspond to the great circle geodesics
terminating at $\mathsf{NP}$, the north pole.

The collision time can be computed explicitly. On the sphere side, suppose $(u_0,\eta_0) \in S^*(\mathbb{S}_{\neq \mathsf{NP}}^d)$
generates a great circle $\varphi_0(s)$ through $\mathsf{NP}$,
and let $s_*>0$ denote the smallest $s \in (0,2\pi]$ satisfying
$\varphi_0(s)=\mathsf{NP}$. Explicitly,
$$s_* \coloneqq  \operatorname{Arg}\big((u_0+i \eta_0)_{d+1}\big) \in (0,2\pi].$$
 The collision time of the corresponding collision orbit $\gamma_0$ is $t_* \coloneqq t(s_*)$, where $t(\bullet)$ is defined in \eqref{eq:timechange2}.                 
\end{rem}
Moser's regularization adds the point $\mathsf{NP}$ along with the missing co-fiber to
$S^*(\mathbb{S}_{\neq \mathsf{NP}}^d)$ and thus compactifies
$\Sigma_E$ (see Figure \ref{fig:5}). In order to do this rigorously, we `patch' the
behavior at the south pole to the north pole. Defining
$\mathsf{SP} \coloneqq -\mathsf{NP}=(0,0,\ldots,-1)$, observe the
diagram
\begin{equation}\label{eq:newdiagram}
\begin{tikzcd}
	{(T^*\mathbb{R}^d)\setminus \{\xi=0\}} & {T^*(\mathbb{S}_{\neq \mathsf{SP},\mathsf{NP}}^d)} \\
	{(T^*\mathbb{R}^d)\setminus \{\xi=0\}} & {T^*(\mathbb{S}_{\neq \mathsf{SP},\mathsf{NP}}^d)}
	\arrow["{\mathcal{M}_E}", from=1-1, to=1-2]
	\arrow["{\mathcal{I}_E}"', from=1-1, to=2-1]
	\arrow["{\mathcal{N}}", from=1-2, to=2-2]
	\arrow["{\mathcal{M}_E}", from=2-1, to=2-2]
  \end{tikzcd}\end{equation}
  commutes, where
  \begin{align*}
    \mathcal{N}(u,\eta) \coloneqq &\ (-u,-\eta),\qquad \mathcal{I}_E
                                \coloneqq \mathcal{D}_{p_0^{-2}}\circ
                     R_{-\pi/2}\circ \iota^* \circ
                                R_{-\pi/2},\\
    \iota(x) \coloneqq &\ \frac{x}{|x|^2},\qquad \iota^*(x,\xi) \coloneqq \Big(\frac{x}{|x|^2},|x|^2\xi-2(x\cdot \xi)x\Big).
\end{align*}
Explicitly,
\begin{equation}\label{eq:energyinversion}
\mathcal{I}_E(x,\xi)=\Big(p_0^{-2}\big(-|\xi|^2x+2(x\cdot \xi)\xi \big) ,-p_0^2\frac{\xi}{|\xi|^2}\Big).
\end{equation}
It is easy to see from \eqref{eq:newdiagram} that $\mathcal{I}_E$ is an involution and it takes the set $\Sigma_E \setminus \{(x,0) \colon
|x|=2p_0^{-2}\}$ to itself. Now we define the compactification of $\Sigma_E$:

\begin{defn}\label{def:moser2}
For $E<0$, define Moser's compactified energy hypersurface $\overline{\Sigma_E}$ by
\begin{equation}\label{eq:compact}
\overline{\Sigma_E} \coloneqq \big((\Sigma_E)^{(0)} \sqcup (\Sigma_E)^{(1)}\big)/\sim, \qquad  \big(\mathcal{I}_E(x,\xi),0\big)\sim\big((x,\xi),1\big)\text{ for } \xi \neq 0,
\end{equation}
where $\mathcal{I}_E$ is defined in \eqref{eq:energyinversion} (see Figure \ref{fig:5}). Define
the regularized Moser map $\overline{\mathcal{M}_E}:\overline{\Sigma_E}\to
S^*\mathbb{S}^d$ by
\begin{equation}\label{eq:compactmoser}
\begin{aligned}
\overline{\mathcal{M}_E}\big( (x,\xi),0 \big) &\coloneqq
                                                \mathcal{M}_E(x,\xi),\\
  \overline{\mathcal{M}_E}\big( (x,\xi),1 \big) &\coloneqq
  \mathcal{M}_E\big(\mathcal{I}_E(x,\xi)\big),\text{ when }\xi \neq 0,\\
  \overline{\mathcal{M}_E}\big((x,0),1\big) &\coloneqq
  \big(\mathsf{NP},(2^{-1}p_0^2x,0)\big),\text{ where }|x|=2/p_0^2.
  \end{aligned}
  \end{equation}
  \end{defn}

 \begin{figure}[!htbp]
\centering
\begin{subfigure}{.45\textwidth}
\centering
  \includegraphics[width=0.75\linewidth]{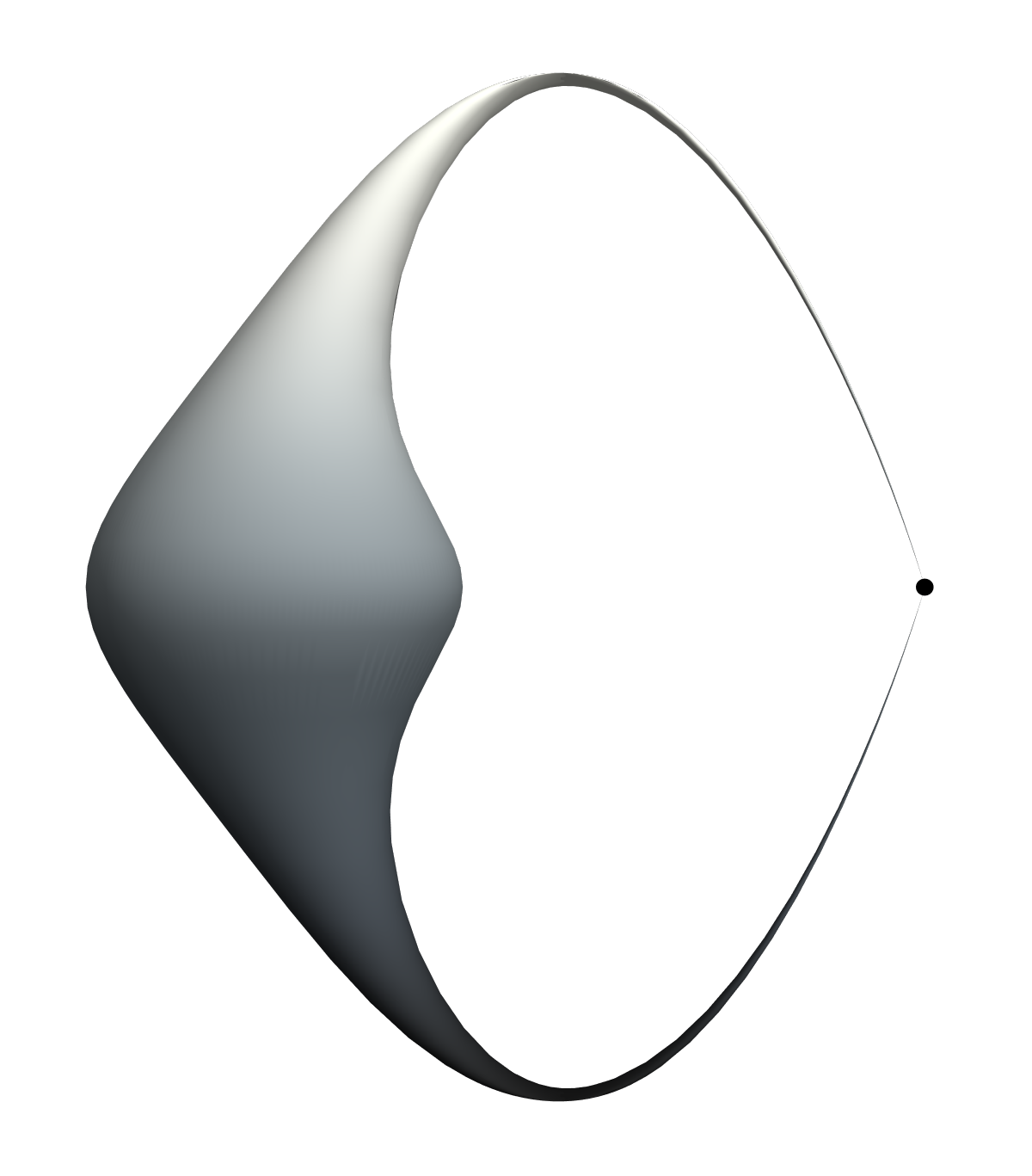}%
  \caption{Compactification of $\Sigma_E$ at the `point' at infinity in Figure \ref{fig:4b}}\label{fig:5a}
  \end{subfigure}%
  \hfill
\begin{subfigure}{.45\textwidth}
  \centering
  \includegraphics[width=.75\linewidth]{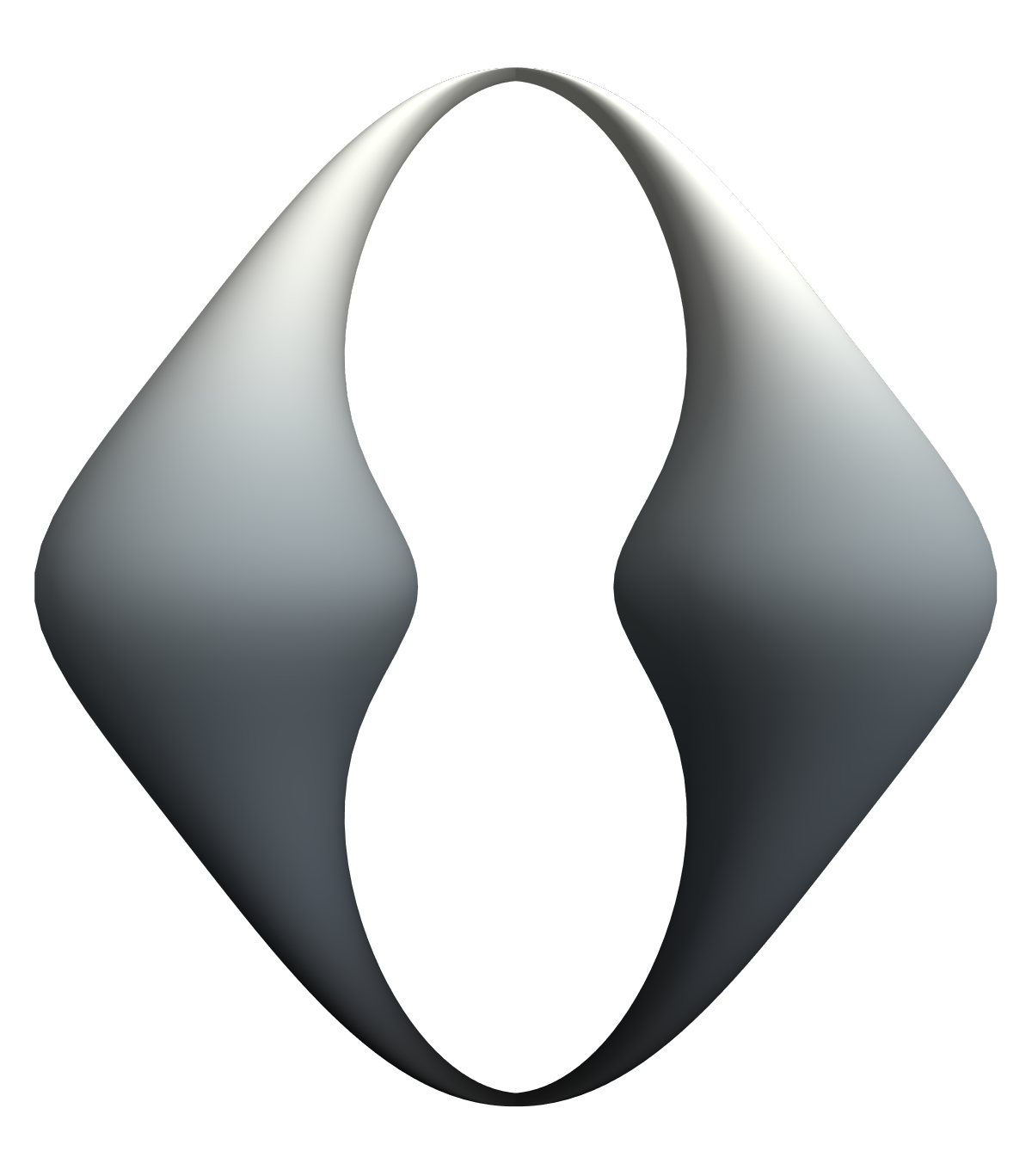}%
  \caption{The space $\overline{\Sigma_E}$ viewed as the `blow-up' of infinity in Figure \ref{fig:5a}}\label{fig:5b}
  \end{subfigure}%
  \caption{Visualization of $\overline{\Sigma_E}$ when $E=-1/2$}\label{fig:5}%
  \end{figure}

  \begin{rem}\label{rem:moser2}
One can show $\overline{\mathcal{M}_E}$ is a smooth diffeomorphism, and we
can then define the regularized Moser flow on
$\overline{\Sigma_E}$: for any $t \in \mathbb{R}$,
define $\overline{\Xi_{H}^t}:\overline{\Sigma_E}\to
\overline{\Sigma_E}$ by
\begin{equation}\label{eq:reghamil}
\overline{\Xi_{H}^t} \coloneqq \overline{\mathcal{M}_E}^{-1}\circ \Phi_{\mathbb{S}^d}^{s(t)}\circ  \overline{\mathcal{M}_E},
\end{equation}
where $\Phi_{\mathbb{S}^d}^{\bullet}$ denotes the cogeodesic
flow on $S^*\mathbb{S}^d$ and $s(t)$ is the inverse of $t(s)$
defined in \eqref{eq:timechange}. More precisely, keeping track of the dependence on the input, the regularized flow $\overline{\Xi_{H}^t}$ is defined as 
\begin{equation}\label{eq:reghamil2}
\overline{\Xi_{H}^t}(z_0) \coloneqq (\overline{\mathcal{M}_E}^{-1}\circ \Phi_{\mathbb{S}^d}^{s_{z_0}(t)}\circ  \overline{\mathcal{M}_E})(z_0).
\end{equation}
Indeed, the time change depends on the initial point of the
orbit $z_0 \in \overline{\Sigma_E}$ (or, more explicitly,
$\overline{\mathcal{M}_E}(z_0)$), as seen from
\eqref{eq:timechange2}. Regardless, one can show the usual semigroup
property
$$\overline{\Xi_{H}^{t'+t}}=\overline{\Xi_{H}^{t'}}\circ
\overline{\Xi_{H}^{t}}.$$
We end this remark by noting that had we defined the regularized
flow without the time change
$$\Phi_H^s:\overline{\Sigma_E}\to
\overline{\Sigma_E},\quad \Phi_H^s \coloneqq
\overline{\mathcal{M}_E}^{-1}\circ \Phi_{\mathbb{S}^d}^{s}\circ
\overline{\mathcal{M}_E}, $$
we obtain a fundamentally different flow. Indeed, one can easily show that $\overline{\mu}$ is a measure on $\overline{\Sigma_E}$ invariant under $\Phi_H^{\bullet}$ if and only if $(\overline{\mathcal{M}_E })_*\overline{\mu}$ is invariant under the cogeodesic flow on $S^*\mathbb{S}^d$, but the same is \textit{not} true for the regularized Moser flow $\overline{\Xi_H^{\bullet}}$, as seen in Lemma \ref{lem:invsig}.
\end{rem}

\begin{defn}\label{def:smooth}
Define the inclusion
\begin{equation}\label{eq:inclusion}
i_{\Sigma_E}:\Sigma_E \to \overline{\Sigma_E},\quad i_{\Sigma_E}(x,\xi)=\big((x,\xi),0\big). 
\end{equation}
We say that a function $a \in
C^{\infty}(\Sigma_E)$ \textit{extends smoothly to
}$\overline{\Sigma_E}$ if there exists a function $\overline{a}
\in C^{\infty}(\overline{\Sigma_E})$ such that
$$\overline{a}\circ i_{\Sigma_E}=a.$$
\end{defn}

\begin{rem}\label{rem:incl}

If $(x,\xi) \in \Sigma_E$ is
on a non-collision orbit,
it is easy to see from definitions and Theorem \ref{theo:moser} that
$$\overline{\Xi_{H}^t}\big((x,\xi),0\big)=
i_{\Sigma_E}\big(\Xi_{H}^t(x,\xi)\big)$$
for any $t \in \mathbb{R}$, where $\Xi_{H}^t$ is the
non-regularized Hamiltonian flow.
\end{rem}

\begin{rem}\label{rem:bars}
 The function $\overline{a}$
is unique by \eqref{eq:compact}. Equivalently, $a \in
C^{\infty}(\Sigma_E)$ extends smoothly to $\overline{\Sigma_E}$
if and only if
\begin{equation}\label{eq:ainv}
(x,\xi) \mapsto
a_{\mathrm{inv}}(x,\xi) \coloneqq a\big(\mathcal{I}_E(x,\xi)\big)=a\Big(p_0^{-2}\big(-|\xi|^2x+2(x\cdot
\xi)\xi \big) ,-p_0^2\frac{\xi}{|\xi|^2}\Big)
\end{equation}
extends to a smooth function on $\Sigma_E$ over $\xi=0$. In this case, $\overline{a} \in C^{\infty}(\overline{\Sigma_E})$ equals
\begin{equation}\label{eq:barr}
\begin{aligned}
\overline{a}\big( (x,\xi),0 \big) &\coloneqq
                                               a(x,\xi),\\
  \overline{a}\big( (x,\xi),1 \big) &\coloneqq
  a_{\mathrm{inv}}(x,\xi).
\end{aligned}
\end{equation}
Examples of symbols that extend smoothly to $\overline{\Sigma_E}$ include $a \in C_c^{\infty}(T^*\mathbb{R}^d)$, as well as symbols that are Schwartz functions of $x$ (or $\xi$). Lastly, we note that in connection with the Moser map, $a \in
C^{\infty}(\Sigma_E)$ extends smoothly to $\overline{\Sigma_E}$
if and only if  $a \circ \mathcal{M}_E^{-1}\in C^{\infty}(S^*(\mathbb{S}_{\neq \mathsf{NP}}^d))$ extends smoothly to $S^*\mathbb{S}^d$.
\end{rem}

\begin{lem}\label{lem:invsig} Suppose $\overline{\mu}$ is a probability measure on
$\overline{\Sigma_E}$. Then $\overline{\mu}$ is invariant under
the Moser-regularized flow if and only if there exists a unique
probability measure $\nu$ on $S^*\mathbb{S}^d$ invariant under the
cogeodesic flow such that
\begin{equation}\label{eq:weirdmeas}
(\overline{\mathcal{M}_E})_*\overline{\mu}=(1-u_{d+1})\nu.
\end{equation}
\end{lem}
\begin{rem}\label{rem:invsig}
We first observe that $(1-u_{d+1})\nu$ is a probability measure on $S^*\mathbb{S}^d$ when $\nu$ is a probability measure on $S^*\mathbb{S}^d$ invariant under the
cogeodesic flow. Indeed, by flowing by $\pi$,
$$\int_{S^*\mathbb{S}^d}u_{d+1}d\nu=\int_{S^*\mathbb{S}^d}u_{d+1}d(\Phi^{\pi})_*\nu=\int_{S^*\mathbb{S}^d}(-u_{d+1})d\nu
\implies \int_{S^*\mathbb{S}^d}u_{d+1}d\nu=0. $$
The uniqueness statement also follows from flow invariance. Suppose $\nu'$ is another probability measure on $S^*\mathbb{S}^d$ invariant under the
cogeodesic flow such that
$$(1-u_{d+1})\nu=(1-u_{d+1})\nu'.$$
We claim $\nu=\nu'$. Indeed, it suffices to show
$\nu(S_{\mathsf{NP}}^*\mathbb{S}^d)=\nu'(S_{\mathsf{NP}}^*\mathbb{S}^d)=0$
since
$$\nu|_{S^*\mathbb{S}^d \setminus
  S_{\mathsf{NP}}^*\mathbb{S}^d}=\nu'|_{S^*\mathbb{S}^d \setminus
  S_{\mathsf{NP}}^*\mathbb{S}^d},$$
due to the fact $1-u_{d+1}>0$ on this restriction. To prove
$\nu(S_{\mathsf{NP}}^*\mathbb{S}^d)=\nu'(S_{\mathsf{NP}}^*\mathbb{S}^d)=0$,
we use the flow invariance. For distinct $s_1,\ldots, s_L \in
(0,\pi)$, we see that the sets
$\Phi^{s_j}(S_{\mathsf{NP}}^*\mathbb{S}^d)$ are pairwise disjoint because
they each lie in the pairwise disjoint level sets
$\{u_{d+1}=\cos s_j\}$. By flow invariance,
$\nu(S_{\mathsf{NP}}^*\mathbb{S}^d)=\nu(\Phi^{s_j}(S_{\mathsf{NP}}^*\mathbb{S}^d))$,
so we have
$$1 =\nu(S^*\mathbb{S}^d) \geq \nu \Big(
\bigsqcup_{j=1}^L\Phi^{s_j}(S_{\mathsf{NP}}^* \mathbb{S}^d)\Big)=\sum_{j=1}^L\nu(\Phi^{s_j}(S_{\mathsf{NP}}^*\mathbb{S}^d))=L\cdot\nu(S_{\mathsf{NP}}^*\mathbb{S}^d). $$
Letting $L \to \infty$ gives $\nu(S_{\mathsf{NP}}^*\mathbb{S}^d)=0$. The same proof works for $\nu'$.
\end{rem}
\begin{proof}[Proof of Lemma \ref{lem:invsig}]
Suppose $\overline{\mu}$ is a probability measure on
$\overline{\Sigma_E}$ invariant under
the Moser-regularized flow. For $b \in C(S^*\mathbb{S}^d)$,
\begin{align*}
\int_{S^*\mathbb{S}^d}b\ d(\overline{\mathcal{M}_E})_*\overline{\mu} &=\int_{\overline{\Sigma_E}}b\circ \overline{\mathcal{M}_E} d\overline{\mu}.
\end{align*}
We now view this integral as an integral over oriented
regularized Moser orbits. Define
$\mathcal{H}(\overline{\Sigma_E})=\overline{\Sigma_E}/\sim$
where we quotient out by points on the same (oriented)
regularized Moser orbit. By \eqref{eq:reghamil2},
$$\mathcal{H}(\overline{\Sigma_E}) \cong (S^*\mathbb{S}^d)/\mathbb{S}^1=\operatorname{SO}(d+1)/(\operatorname{SO}(2) \times \operatorname{SO}(d-1))=\widetilde{\textbf{Gr}}(2,d+1), $$
where $\widetilde{\textbf{Gr}}(2,d+1)$ is the (compact) oriented
Grassmannian manifold (i.e. the double cover of
$\textbf{Gr}(2,d)$) and ``$\cong$'' is the
induced Moser map
$$\widetilde{\mathcal{M}_E}:\mathcal{H}(\overline{\Sigma_E}) \to
(S^*\mathbb{S}^d)/\mathbb{S}^1,\quad [z] \mapsto [\overline{\mathcal{M}_E}(z)]. $$
Denote $\pi:\overline{\Sigma_E}\to
\mathcal{H}(\overline{\Sigma_E})$ the projection, then the
disintegration theorem gives
\begin{equation}\label{eq:disint}
\int_{\overline{\Sigma_E}}b\circ \overline{\mathcal{M}_E} d\overline{\mu}=\int_{\mathcal{H}(\overline{\Sigma_E})}\Big(\int_{\pi^{-1}(\overline{\gamma})}b\circ \overline{\mathcal{M}_E}d \mu_{\overline{\gamma}} \Big)d(\pi_*\overline{\mu})(\overline{\gamma})
\end{equation}
where $\mu_{\overline{\gamma}}$ are probability measures on
$\overline{\Sigma_E}$ such that $\operatorname{supp}
\mu_{\overline{\gamma}}\subseteq \pi^{-1}(\overline{\gamma})$
for $\pi_*\overline{\mu}$-almost all $\overline{\gamma}\in
\mathcal{H}(\overline{\Sigma_E})$ (see \cite[III-70]{DM78} for
the disintegration theorem). Since $\overline{\mu}$ is invariant
under the regularized Moser flow, we see that
$\mu_{\overline{\gamma}}$ is also invariant under the
regularized Moser flow for $\pi_*\overline{\mu}$-almost all
$\overline{\gamma}$ (as shown by either the uniqueness statement
of the disintegration theorem or by considering test functions
$\psi(\pi(\bullet))\phi(\bullet)$ where $\psi \in
C(\mathcal{H}(\overline{\Sigma_E})),\phi\in
C(\overline{\Sigma_E})$). Then, for $\pi_*\overline{\mu}$-almost
all $\overline{\gamma}$,
\begin{equation}\label{eq:prechange}
\int_{\pi^{-1}(\overline{\gamma})}a d\mu_{\overline{\gamma}}=\frac{p_0^3}{2\pi}\int_0^{2\pi/p_0^3}a(\overline{\gamma}(t))dt
\end{equation}
for all $a \in C(\overline{\Sigma_E})$. We apply the change of
variables $t \mapsto t(s)$ to \eqref{eq:prechange} and set $a =b \circ \overline{\mathcal{M}_E}$, where $t(s)$
is defined in \eqref{eq:timechange2} (with initial point
$(u_0,\eta_0)=\overline{\mathcal{M}_E}(\overline{\gamma}(0))$) to obtain
\begin{equation}\label{eq:prechange2}
\int_{\pi^{-1}(\overline{\gamma})}b \circ \overline{\mathcal{M}_E}^{-1} d\mu_{\overline{\gamma}}=\frac{1}{2\pi}\int_0^{2\pi}b(\varphi(s)) (1-u_{d+1}(\varphi(s)))ds,
\end{equation}
where $\varphi \coloneqq
\widetilde{\mathcal{M}_E}(\overline{\gamma})$ is the
corresponding cogeodesic on $S^*\mathbb{S}^d$ corresponding to
$\overline{\gamma}$ on $\overline{\Sigma_E}$ via
$\overline{\mathcal{M}_E}$, and $u_{d+1}(\bullet)$ is the
$(d+1)$-component of the base variable in $\mathbb{S}^d \subset
\mathbb{R}^{d+1}$. Now \eqref{eq:prechange2} can be written as
\begin{equation}\label{eq:prechange3}
\int_{\pi^{-1}(\overline{\gamma})}b \circ \overline{\mathcal{M}_E}^{-1} d\mu_{\overline{\gamma}}=\int_{\pi_{\mathbb{S}^d}^{-1}(\varphi)}b\cdot (1-u_{d+1})d\nu_{\varphi}
\end{equation}
where $\pi_{\mathbb{S}^d}:S^*\mathbb{S}^d \to
S^*\mathbb{S}^d/\mathbb{S}^1$ and $\nu_{\varphi}$ is the
cogeodesic flow invariant probability measure on $S^*\mathbb{S}^d$ supported
on $\pi_{\mathbb{S}^d}^{-1}(\varphi)$. Now substituting
\eqref{eq:prechange3} into \eqref{eq:disint} and pushing forward
the integral via $\overline{\mathcal{M}_E}$, we obtain
\begin{align}
\int_{S^*\mathbb{S}^d}b\ d(\overline{\mathcal{M}_E})_*\overline{\mu}=\int_{(S^*\mathbb{S}^d)/\mathbb{S}^1}\Big(\int_{\pi_{\mathbb{S}^d}^{-1}(\varphi)}b\cdot (1-u_{d+1}) d \nu_{\varphi} \Big)d\big((\widetilde{\mathcal{M}_E})_*\pi_*\overline{\mu}\big)(\varphi).\label{eq:leave}
\end{align}
Finally, if we define
$$\nu \coloneqq \int_{(S^*\mathbb{S}^d)/\mathbb{S}^1}
\nu_{\varphi}\ d\big((\widetilde{\mathcal{M}_E})_*\pi_*\overline{\mu}\big)(\varphi), $$
we see that $\nu$ is a probability measure invariant under the
cogeodesic flow by construction, and \eqref{eq:leave} shows
$(\overline{\mathcal{M}_E})_*\overline{\mu}=(1-u_{d+1})\nu$, as
desired.

This argument is completely reversible. Indeed, for the converse, one starts by disintegrating $\nu$ along cogeodesics and ends with an equation similar to \eqref{eq:disint} by applying the reverse change of variables from \eqref{eq:prechange2} to \eqref{eq:prechange}.
\end{proof}

\subsection{Background: Spectral theory of \texorpdfstring{$\widehat{H}_{\hbar}$}{H h} and the Fock
  map}\label{subsec:spec}\par
As noted in the introduction, the operator
$\widehat{H}_{\hbar}:L^2(\mathbb{R}^d) \to L^2(\mathbb{R}^d)$ is
initially defined on $C_c^{\infty}(\mathbb{R}^d)$ thanks to
Hardy's inequality. Using Hardy's inequality further (along with
the Peter-Paul inequality), one can show $\widehat{H}_{\hbar}$
is semibounded below (see \cite[Chapter 8, \S 7]{T11} and
\cite[Example 3.2.11]{LN25}). Thus $\widehat{H}_{\hbar}$ extends
to a self-adjoint Friedrichs extension defined on $H^2(\mathbb{R}^d)$ (\cite[Chapter
8, \S 7, Proposition 7.2]{T11}).
\par Because $-|x|^{-1} \in
L^2+L^{\infty}$ and $-|x|^{-1}\to 0$ as $x \to \infty$, theorems
of Kato imply $\widehat{H}_{\hbar}$ is a relatively compact
perturbation of $-\Delta$. Consequently, the essential spectrum
of $\widehat{H}_{\hbar}$ coincides with that of $-\Delta$, which
is $[0,\infty)$ (\cite[Theorem V.5.7]{K95}). Thus the spectrum
of $\widehat{H}_{\hbar}$ in $(-\infty,0)$ consists only of
isolated eigenvalues with finite multiplicities, possibly
infinitely many accumulating at $0$. The latter indeed happens
in the Coulomb case since $-|x|^{-1}$ is attractive and decaying slowly
at infinity. This can be proved by the min-max principle (see
\cite[Theorem XIII.6]{RS4}).
\par Because the potential is radial, one can use the $\operatorname{SO}(d)$-symmetry and separate variables. This leads to the study of certain ordinary differential equations which one can solve explicitly (see \cite[\S 2.1.2]{L25}). Overall, one obtains
\begin{equation}\label{eq:energy}
\operatorname{spec}\widehat{H}_{\hbar} =\Big\{E_{\hbar}(N) \coloneqq  -\frac{1}{2\hbar^2(N+\frac{d-1}{2})^2} \mid N=0,1,\ldots\Big\}\sqcup [0,\infty),
\end{equation}
where each eigenvalue $E_{\hbar}(N)$ has multiplicity
$\binom{N+d-1}{d-1}+\binom{N+d-2}{d-1}$. We define the eigenspace
\begin{equation}\label{eq:eigenspace}
\mathcal{E}_{E_{\hbar}(N)} \coloneqq \{ \psi \in H^2(\mathbb{R}^d) \mid \widehat{H}_{\hbar}\psi=E_{\hbar}(N)\psi\}.
\end{equation}Each eigenspace 
$\mathcal{E}_{E_{\hbar}(N)}$ is spanned by (explicit) functions of the form
\begin{equation}\label{eq:formeq}
C_{\hbar,N,\ell}e^{-|\widetilde{x}|}p_{N-\ell}(|\widetilde{x}|)q_{\ell}(x), \quad \widetilde{x}\coloneqq \frac{\sqrt{-2E_{\hbar}(N)}}{\hbar}x
\end{equation}
where $C_{\hbar,N,\ell}$ is an $L^2$-normalization constant,
$p_{N-\ell}$ is a polynomial of degree $N-\ell$ with
$p_{N-\ell}(0)=1$, and $q_{\ell}$ is a homogeneous polynomial of
degree $\ell$ (see Figure \ref{fig:eig} for some intensity plots and \cite[(2.18)]{L25} for the explicit formulas).
\begin{figure}[!htbp]
\centering
  \includegraphics[width=0.6\linewidth]{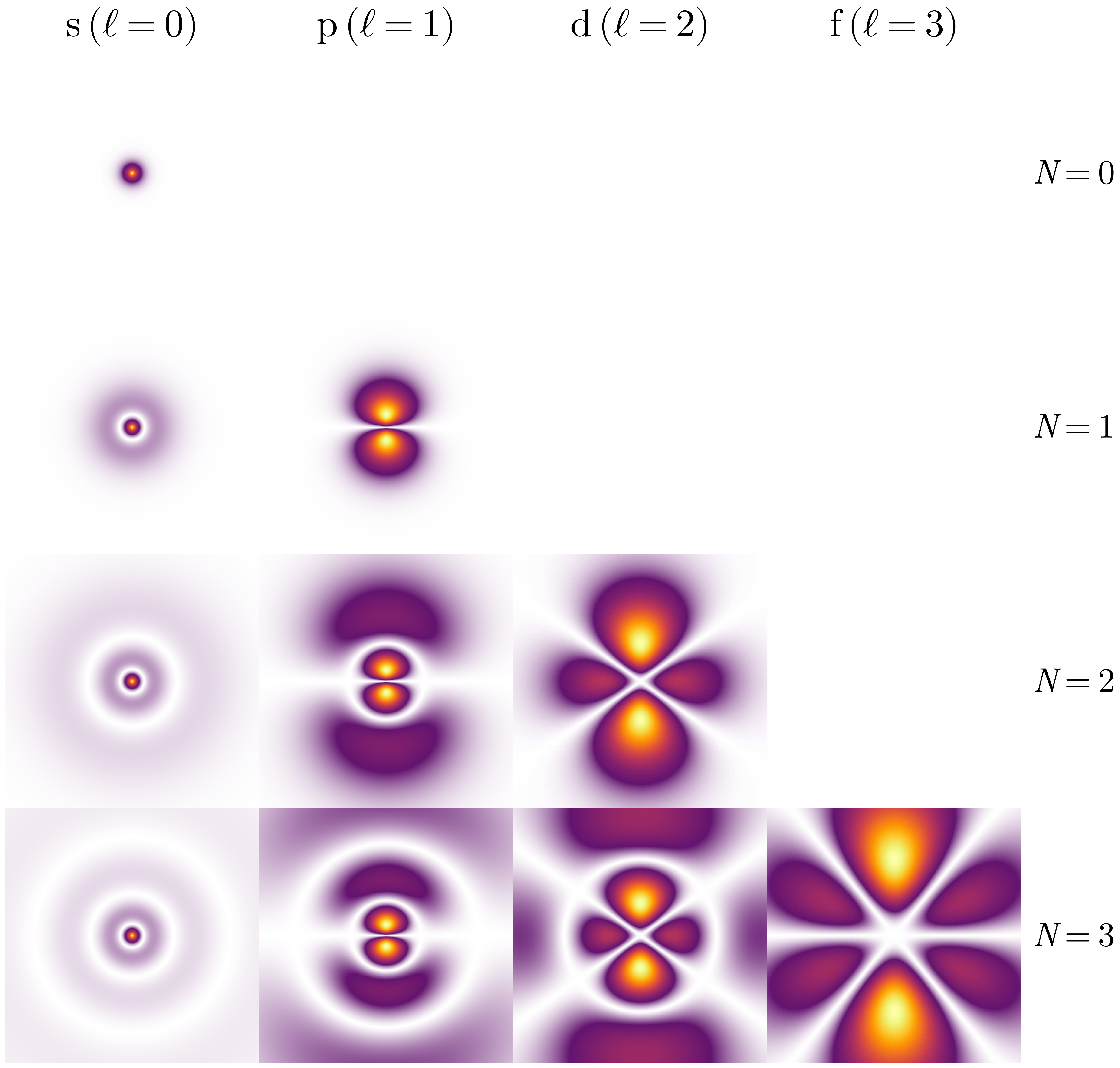}%
  \caption{Intensity plot of the eigenfunctions \eqref{eq:formeq} in $d=3$ on $[-20,20] \times [-20,20]$ in the $xz$-plane, where $\hbar=1$ and the magnetic number $m=0$.}\label{fig:eig}
  \end{figure}

Consequently, one can show
\begin{equation}\label{eq:reg}
\mathcal{E}_{E_{\hbar}(N)} \subset H^{\frac{d}{2}+1-0}(\mathbb{R}^d) \cap W^{1,\infty}(\mathbb{R}^d) \cap \Big\{\psi \in C^{\infty}(\mathbb{R}^d \setminus \{0\}): \sup_{|x|\geq 1}\langle x \rangle^M|\partial^{\alpha}\psi|<\infty, \begin{array}{l}\forall M>0\\ \forall \alpha \in \mathbb{Z}_{\geq 0}^d\end{array}\Big\}.
\end{equation}
One can show this by analyzing the specific form of the eigenfunctions in \eqref{eq:formeq}, or one can use resolvent estimates (\cite[(1.6),(1.7)]{T22}) along with an Agmon-type theorem for the exponential decay (\cite[Theorem 3.4]{HS96}) and elliptic regularity for the smoothness away from the origin. In general, one cannot do much better than \eqref{eq:reg}, as the function $e^{-|x|}$ fails to be in $H^{d/2+1}\cup C^1$, as shown by its explicit Fourier transform.
\par The eigenspace multiplicity of
$\mathcal{E}_{E_{\hbar}(N)}$
\begin{equation}\label{eq:eigmult2}
\operatorname{dim} \mathcal{E}_{E_{\hbar}(N)}=\binom{N+d-1}{d-1}+\binom{N+d-2}{d-1}
\end{equation}
happens to be the \textit{same} as
the eigenspace multiplicity of the degree $N$ spherical
harmonics of $-\Delta_{\mathbb{S}^d}$. This is not a
coincidence, as it is a corollary of a theorem of Fock:
\begin{theo}[\cite{F35}]\label{th:fock} Suppose $E_{\hbar}<0$ is an eigenvalue
of $\widehat{H}_{\hbar}$ with $\hbar >0$. Then the Fock map
$\mathcal{V}_{\hbar,E_{\hbar}}:\mathcal{E}_{E_{\hbar}} \to
\mathcal{E}_{\mathbb{S}^d}(N)$ defined by
\begin{align}
\mathcal{V}_{\hbar,E_{\hbar}}\coloneqq&\ \widehat{\omega}\circ
  \mathcal{F}_{\hbar}\circ
  \widehat{\mathcal{D}}_{\sqrt{-2E_{\hbar}}}\circ
  \frac{1}{\sqrt{2}}\Big\langle
  \frac{1}{\sqrt{-2E_{\hbar}}}\hbar
          D\Big\rangle \label{eq:FOCK}\\
=&\  \widehat{\omega}\circ \mathsf{M}_{\frac{\langle x \rangle}{\sqrt{2}}}\circ 
  \mathcal{F}_{\hbar}\circ
  \widehat{\mathcal{D}}_{\sqrt{-2E_{\hbar}}} \label{eq:FOCK2}
\end{align}
is a unitary map onto $\mathcal{E}_{\mathbb{S}^d}(N)$, the space
of spherical harmonics of degree $N$, where $N$ satisfies
$E_{\hbar}^{-1}=-2\hbar^2(N+\frac{d-1}{2})^2$. The map
$\mathcal{F}_{\hbar}$ is the (unitary) semiclassical Fourier transform,
and the remaining maps are defined by
\begin{equation}\label{eq:fockmaps}
\begin{gathered}
\widehat{\omega}:L^2(\mathbb{R}^d) \to L^2(\mathbb{S}^d),\quad
\widehat{\omega}[f](u) \coloneqq
\Big(\frac{|\omega^{-1}(u)|^2+1}{2}\Big)^{d/2}f\big(\omega^{-1}(u)\big)\\
\widehat{\mathcal{D}}_k:L^2(\mathbb{R}^d)\to
L^2(\mathbb{R}^d),\quad
\widehat{\mathcal{D}}_k[f](x)=k^{-d/2}f(k^{-1}x)\\
\langle k D \rangle :H^1(\mathbb{R}^d) \to
L^2(\mathbb{R}^d),\quad \langle k D \rangle \coloneqq
\sqrt{1-k^2\Delta},\\
\mathsf{M}_{\frac{\langle x \rangle}{\sqrt{2}}}:L^2(\mathbb{R}^d)\to \langle x \rangle L^2(\mathbb{R}^d),\quad \mathsf{M}_{\frac{\langle x \rangle}{\sqrt{2}}}[f](x) \coloneqq \frac{\langle x \rangle}{\sqrt{2}}f(x)=\sqrt{\frac{1+|x|^2}{2}}f(x),
\end{gathered}
\end{equation}
for any $k>0$, where $\omega^{-1}$ is defined in \eqref{eq:stereo}.
\end{theo}

\begin{rem}
The Fock map \eqref{eq:FOCK} compares nicely with the classical
Moser map \eqref{eq:MOSER}. Indeed, the maps
$\omega^*,R_{-\pi/2},\mathcal{D}_{p_0}$ of \eqref{eq:MOSER}
quantize, in the sense of semiclassical (unitary) Fourier
integral operators, to
$\widehat{\omega},\mathcal{F}_{\hbar},\widehat{\mathcal{D}}_{\sqrt{-2E_{\hbar}}}$
of \eqref{eq:FOCK}, respectively. At first glance, it may seem
that the peculiar operator $ \frac{1}{\sqrt{2}}\langle
\frac{1}{\sqrt{-2E_{\hbar}}}\hbar D\rangle$ (or, equivalently,
$\mathsf{M}_{\frac{\langle x \rangle}{\sqrt{2}}}$) in the Fock map is a
consequence of the fact that the Moser map is \textit{not} a
symplectomorphism due to the non-symplectic dilation
$\mathsf{S}$. While this could be a first clue, it cannot
be the whole story, as the Moser map \textit{is} a symplectomorphism in
the special case $p_0=1$.

Regardless of the value of $p_0$, $\mathcal{M}_E|_{\Sigma_E}:(\Sigma_E,-p_0 x\cdot d\xi|_{\Sigma_E}) \to (S^*(\mathbb{S}_{\neq \mathsf{NP}}^d),\eta\cdot du|_{S^*(\mathbb{S}_{\neq \mathsf{NP}}^d)}) $ is
a strict contact transformation by \eqref{eq:sympl}. Indeed, if
$X_{H} \coloneqq \xi \cdot \partial_x-\frac{x}{|x|^3}\cdot
\partial_{\xi}$ denotes the Hamiltonian vector field of $H$ and
$\alpha \coloneqq -p_0 x\cdot d\xi|_{\Sigma_E}$, then
\begin{equation}\label{eq:contact}
\alpha(X_{H})= -p_0x \cdot
\Big(-\frac{x}{|x|^3}\Big)=\frac{p_0}{|x|}=p_0^3\frac{|\xi|^2/p_0^2+1}{2} \quad \text{on }\Sigma_E,
\end{equation}
which never vanishes. Equation \eqref{eq:contact} not only explains the time change factor \eqref{eq:timechange}, but, interestingly, the semiclassical quantization of the square root of \eqref{eq:contact} is $\frac{1}{\sqrt{2}}\langle
\frac{1}{\sqrt{-2E_{\hbar}}}\hbar D\rangle$, up to the square root of the prefactor $p_0^{3}$, which is evidently coming from Kepler's third law \eqref{eq:kepler3}.

The operator $\frac{1}{\sqrt{2}}\langle
\frac{1}{\sqrt{-2E_{\hbar}}}\hbar D\rangle$ is rather
miraculously an $L^2$-isometry on $\mathcal{E}_{E_{\hbar}}$, as
explained below in Remark \ref{rem:unitary} by the quantum virial theorem. Generally,
$\frac{1}{\sqrt{2}}\langle \frac{1}{\sqrt{-2E_{\hbar}}}\hbar
D\rangle:L^2 \to H^{-1}$, so there is no hope for a similar
statement for $L^2$ functions not in $\mathcal{E}_{E_{\hbar}}$.

In any event, it is worth noting that Fock himself did not think of this map as a quantization in this way, as his work \cite{F35} predates Moser's \cite{M70}.
\end{rem}
Theorem \ref{th:fock} was originally proved when $d=3$ in \cite{F35} and
other expositions can be found in \cite{BI66}, \cite[Chapter
9]{RC21}, and in general dimension in \cite[\S 2.2.2]{L25}. The main parts of the proof can
be broken up as follows:
\begin{itemize}
\item $\mathcal{V}_{\hbar,E_{\hbar}}$ is $L^2$-norm preserving (and hence injective)
by a consequence of the quantum virial theorem: Lemma 
\ref{lem:qvt}. See Remark \ref{rem:unitary} for the proof.
\item $\mathcal{V}_{\hbar,E_{\hbar}}$ maps into $\mathcal{E}_{\mathbb{S}^d}(N)$ due to a myriad of different arguments. Typically, one starts by applying the Fock map to both sides of the Helmholtz equation \eqref{eq:EIGGG} and arrives at an integral equation involving a Riesz-type operator on the sphere. One then shows that this integral equation is actually a Helmholtz equation of a function of $-\Delta_{\mathbb{S}^d}$ by either \begin{itemize}
\item Green's identities \cite[pp. 333]{BI66} (or, relatedly,
with layer potential formulas for the sphere \cite[Chapter 11,
(11.35)]{T11}),
\item a group theoretic approach with Schur's lemma
\cite[pp. 285]{RC21},
\item or using the uniqueness of the Dirichlet problem on the ball \cite[\S 2.2.2]{L25}.
\end{itemize}  
\end{itemize}
 As an isometry between finite dimensional vector spaces of the same dimension, the above bullet points show $\mathcal{V}_{\hbar,E_{\hbar}}$ is unitary.
\begin{lem}[Coulomb Virial Theorem]\label{lem:qvt} We have commutator identity
\begin{equation}\label{eq:comm}
[\widehat{H}_{\hbar},x \cdot \nabla_x]=-\hbar^2\Delta-\frac{1}{|x|},
\end{equation}
as continuous operators $H^2(\mathbb{R}^d) \to L^2(\mathbb{R}^d)$, where the commutator is initially defined with both compositions interpreted in $\mathcal{S}'$ and $d\geq 3$.
\end{lem}
\begin{proof}
Both sides make sense as continuous operators $H^2(\mathbb{R}^d) \to \mathcal{S}'(\mathbb{R}^d)$ since $\frac{1}{|x|}(x \cdot \nabla_x)=\frac{x}{|x|}\cdot \nabla_x:H^2 \to L^2$ and Hardy's inequality gives the embedding $\frac{1}{|x|}:H^2 \to L^2$. On $C_c^{\infty}(\mathbb{R}^d)$, we have $$[|x|^{-1},x\cdot \nabla_x]=|x|^{-1},\qquad [-\Delta,x\cdot \nabla_x]=-2\Delta.$$ We see that \eqref{eq:comm} is true on $C_c^{\infty}$ and hence on $H^2$ by density. The right hand side is in $L^2$ by Hardy's inequality, so we are done. 
\end{proof}
\begin{rem}\label{rem:unitary} An easy consequence of the above shows that
\begin{equation}\label{eq:comeig}
[\widehat{H}_{\hbar}-E_{\hbar},x \cdot \nabla_x]=-\frac{\hbar^2}{2}\Delta+E_{\hbar} \quad \text{on }\mathcal{E}_{E_{\hbar}}.
\end{equation}
If $\psi \in \mathcal{E}_{E_{\hbar}}$, then one can easily show
$r \partial_r\psi=x\cdot \nabla_x\psi \in H^1$ since $\psi\in
\mathcal{S}+H_{\mathrm{loc},\mathrm{comp}}^2$ by \eqref{eq:reg}. Using
Lemma \ref{lem:qvt}, one can show $\Delta r \partial_r \psi \in L^2$ and hence
\begin{equation}\label{eq:reg2}
r\partial_r\psi \in H^2.\footnote{This can also be deduced by the explicit form of $\psi$ in \eqref{eq:formeq}}
\end{equation}
Consequently, for $\psi \in \mathcal{E}_{E_{\hbar}}$, by using the adjoint of $\widehat{H}_{\hbar}-E_{\hbar}$, we have
\begin{align}
\Big\langle
  \Big(-\frac{\hbar^2}{2}\Delta+E_{\hbar}\Big)\psi,\psi
  \Big\rangle=0 &\implies \Big\langle
  \Big(-\frac{\hbar^2}{2}\Delta-E_{\hbar}\Big)\psi,\psi
  \Big\rangle=-2E_{\hbar}\langle \psi,\psi \rangle\nonumber \\
  &\implies \Big\lVert  \frac{1}{\sqrt{2}}\Big\langle \frac{1}{\sqrt{-2E_{\hbar}}}\hbar D\Big\rangle \psi \Big\rVert_{L^2}^2=\lVert \psi \rVert_{L^2}^2, 
\end{align}
proving $\mathcal{V}_{\hbar,E_{\hbar}}$ is an $L^2$-isometry on
$\mathcal{E}_{E_{\hbar}}$.
\end{rem}

\section{Proofs of Theorems \ref{theo:1}, \ref{theo:2}}

\subsection{Proof of Theorem \ref{theo:1}}
Let $E<0$ and take sequences
$\hbar_j \to 0, E_{\hbar_j}\to E$ along with $L^2$-normalized
$E_{\hbar_j}$-eigenfunctions $\Psi_j \in \mathcal{E}_{E_{\hbar_j}}$ of
$\widehat{H}_{\hbar_j}$. Define
\begin{equation}\label{eq:phi}
\phi_j \coloneqq \mathcal{V}_{\hbar_j,E_{\hbar_j}}[\Psi_j]
\end{equation}
be $N_j(N_j+d-1)$-eigenfunctions of $-\Delta_{\mathbb{S}^d}$ where $N_j$
satisfies $E_{\hbar_j}^{-1}=-2\hbar_j^2(N_j+\frac{d-1}{2})^2$. For
notational convenience, define
\begin{equation}\label{eq:numbers}
e_j \coloneqq \sqrt{-2E_{\hbar_j}},\quad p_0 \coloneqq \sqrt{-2E},\quad k_j \coloneqq \frac{1}{\sqrt{N_j(N_j+d-1)}}=\frac{\hbar_j}{\sqrt{(-2E_{\hbar_j})^{-1}-\frac{(d-1)^2}{4}\hbar_j^2}}.
\end{equation}
Using
the definition of $\mathcal{V}_{\hbar_j,E_{\hbar_j}}$ in
\eqref{eq:FOCK2},
\begin{align}
\langle \operatorname{Op}_{\hbar_j}(a)\Psi_{j},\Psi_{j} \rangle &=\langle \operatorname{Op}_{\hbar_j}(a)\mathcal{V}_{\hbar_j,E_{\hbar_j}}^{-1}\phi_{j},\mathcal{V}_{\hbar_j,E_{\hbar_j}}^{-1}\phi_{j} \rangle \nonumber\\
  \overset{\eqref{eq:FOCK2}}&{=}
                                                         \Big\langle
                                                         \operatorname{Op}_{\hbar_j}(a)\widehat{\mathcal{D}
                                                         }_{e_j}^{-1}\mathcal{F}_{\hbar_j}^{-1}\mathsf{M}_{\langle x\rangle/\sqrt{2}}^{-1}\widehat{\omega}^{-1}\phi_j,\widehat{\mathcal{D}
                                                         }_{e_j}^{-1}\mathcal{F}_{\hbar_j}^{-1}\mathsf{M}_{\langle x\rangle/\sqrt{2}}^{-1}\widehat{\omega}^{-1}\phi_j
                                                         \Big\rangle\nonumber\\
  \overset{ \text{unitarity of }\widehat{\mathcal{D}}_{e_j}}&{=}\Big\langle
                                                        \widehat{\mathcal{D}
                                                         }_{e_j} \operatorname{Op}_{\hbar_j}(a)\widehat{\mathcal{D}
                                                         }_{e_j}^{-1}\mathcal{F}_{\hbar_j}^{-1}\mathsf{M}_{\langle x\rangle/\sqrt{2}}^{-1}\widehat{\omega}^{-1}\phi_j,\mathcal{F}_{\hbar_j}^{-1}\mathsf{M}_{\langle x\rangle/\sqrt{2}}^{-1}\widehat{\omega}^{-1}\phi_j
    \Big\rangle\nonumber\\
   \overset{\cite[(11.3.1)]{Z12}}&{=}\Big\langle
                                                     \operatorname{Op}_{\hbar_j}(\mathcal{D}_{e_j^{-1}}^*a)\mathcal{F}_{\hbar_j}^{-1}\mathsf{M}_{\langle x\rangle/\sqrt{2}}^{-1}\widehat{\omega}^{-1}\phi_j,\mathcal{F}_{\hbar_j}^{-1}\mathsf{M}_{\langle x\rangle/\sqrt{2}}^{-1}\widehat{\omega}^{-1}\phi_j
     \Big\rangle\nonumber\\
  &=\Big\langle
                                                     \operatorname{Op}_{k_j}(\mathsf{S}_{\hbar_j/k_j}^*\mathcal{D}_{e_j^{-1}}^*a)\mathcal{F}_{k_j}^{-1}\mathsf{M}_{\langle x\rangle/\sqrt{2}}^{-1}\widehat{\omega}^{-1}\phi_j,\mathcal{F}_{k_j}^{-1}\mathsf{M}_{\langle x\rangle/\sqrt{2}}^{-1}\widehat{\omega}^{-1}\phi_j
                                                         \Big\rangle,\label{eq:almost}
\end{align}
where we used $\operatorname{Op}_{h}(a)=\widehat{\mathcal{D}
}_{c}^{-1}
\operatorname{Op}_{ch}(\mathsf{S}_{c^{-1}}^*a)\widehat{\mathcal{D}}_c$
with $\mathsf{S}_c(x,\xi)\coloneqq (cx,\xi), c>0$ and
$\widehat{\mathcal{D} }_c
\mathcal{F}_{\hbar}^{-1}=\mathcal{F}_{c\hbar}^{-1}$. Now using
the unitarity of the Fourier transform, from \eqref{eq:almost},
we have
\begin{align}
&\langle \operatorname{Op}_{\hbar_j}(a)\Psi_{j},\Psi_{j} \rangle  \overset{\text{unitarity of }\mathcal{F}_{k_j}}{=} \Big\langle
                                                     \mathcal{F}_{k_j}\operatorname{Op}_{k_j}(\mathsf{S}_{\hbar_j/k_j}^*\mathcal{D}_{e_j^{-1}}^*a)\mathcal{F}_{k_j}^{-1}\mathsf{M}_{\langle x\rangle/\sqrt{2}}^{-1}\widehat{\omega}^{-1}\phi_j,\mathsf{M}_{\langle x\rangle/\sqrt{2}}^{-1}\widehat{\omega}^{-1}\phi_j
                                                                  \Big\rangle\nonumber \\
  \overset{\cite[(11.3.1)]{Z12}}&{=}\Big\langle
                                                     \operatorname{Op}_{k_j}(R_{\pi/2}^*\mathsf{S}_{\hbar_j/k_j}^*\mathcal{D}_{e_j^{-1}}^*a)\mathsf{M}_{\langle x\rangle/\sqrt{2}}^{-1}\widehat{\omega}^{-1}\phi_j,\mathsf{M}_{\langle x\rangle/\sqrt{2}}^{-1}\widehat{\omega}^{-1}\phi_j
                                  \Big\rangle\nonumber \\
  \overset{\text{self adj. of  }\mathsf{M}_{\langle x\rangle/\sqrt{2}}^{-1}}&{=}\Big\langle
                                                     \mathsf{M}_{\langle x\rangle/\sqrt{2}}^{-1}\operatorname{Op}_{k_j}(R_{\pi/2}^*\mathsf{S}_{\hbar_j/k_j}^*\mathcal{D}_{e_j^{-1}}^*a)\mathsf{M}_{\langle x\rangle/\sqrt{2}}^{-1}\widehat{\omega}^{-1}\phi_j,\widehat{\omega}^{-1}\phi_j
    \Big\rangle \nonumber \\
   \overset{\cite[(4.4.18)]{Z12}}&{=}\Big\langle\operatorname{Op}_{k_j}\Big(\frac{2}{1+|x|^2}R_{\pi/2}^*\mathsf{S}_{\hbar_j/k_j}^*\mathcal{D}_{e_j^{-1}}^*a\Big)\widehat{\omega}^{-1}\phi_j,\widehat{\omega}^{-1}\phi_j
                                  \Big\rangle+o(1),
\end{align}
since $\operatorname{Op}_{k_j}(\sqrt{2}/\langle x \rangle)=\mathsf{M}_{\langle x\rangle/\sqrt{2}}^{-1}$. Now $e_j^{-1}\to p_0^{-1}$ and $\hbar_j/k_j \to p_0^{-1}$ by \eqref{eq:numbers},
\begin{equation}\label{eq:beej}
\frac{2}{1+|x|^2}R_{\pi/2}^*\mathsf{S}_{\hbar_j/k_j}^*\mathcal{D}_{e_j^{-1}}^*a \to \frac{2}{1+|x|^2}R_{\pi/2}^*\mathsf{S}_{p_0^{-1}}^*\mathcal{D}_{p_0^{-1}}^*a \quad \text{in } S(1).
\end{equation}
Then
$\operatorname{Op}_{k_j}(\frac{2}{1+|x|^2}R_{\pi/2}^*\mathsf{S}_{\hbar_j/k_j}^*\mathcal{D}_{e_j^{-1}}^*a)=\operatorname{Op}_{k_j}(\frac{2}{1+|x|^2}R_{\pi/2}^*\mathsf{S}_{p_0^{-1}}^*\mathcal{D}_{p_0^{-1}}^*a)+o_{L^2
  \to L^2}(1)$ by Calder\'on-Vaillancourt \cite[Theorem 5.1]{Z12}, so 
\begin{align}
\langle \operatorname{Op}_{\hbar_j}(a)\Psi_{j},\Psi_{j} \rangle \overset{\text{unitarity of }\widehat{\omega}}&{=}\Big\langle
                                                     \widehat{\omega}\operatorname{Op}_{k_j}\Big(\frac{2}{1+|x|^2}R_{\pi/2}^*\mathsf{S}_{p_0^{-1}}^*\mathcal{D}_{p_0^{-1}}^*a\Big)\widehat{\omega}^{-1}\phi_j,\phi_j
                                                                                                                \Big\rangle+o(1).\label{eq:uselemma}
\end{align}
Now since $a \in C_c^{\infty}(T^*\mathbb{R}^d)$, we also have
$\frac{2}{1+|x|^2}R_{\pi/2}^*\mathsf{S}_{p_0^{-1}}^*\mathcal{D}_{p_0^{-1}}^*a
\in C_c^{\infty}(T^*\mathbb{R}^d)$. Lemma \ref{lem:ext1} and the identity $\frac{2}{|\omega^{-1}(u)|^2+1}=1-u_{d+1}$ implies that there exists $b \in
C_c^{\infty}(T^*\mathbb{S}^d)$ such
that $$b|_{S^*(\mathbb{S}_{\neq
    \mathsf{NP}}^d)}=(1-u_{d+1})(\mathcal{M}_E^{-1})^*a, \qquad
b|_{S^*\mathbb{S}^d}=(1-u_{d+1})(\overline{\mathcal{M}_E}^{-1})^*\overline{a}$$
and
$$\operatorname{Op}_{k_j}^{\mathbb{S}^d}(b)=\widehat{\omega}\operatorname{Op}_{k_j}\Big(\frac{2}{1+|x|^2}R_{\pi/2}^*\mathsf{S}_{p_0^{-1}}^*\mathcal{D}_{p_0^{-1}}^*a\Big)\widehat{\omega}^{-1}+o_{L^2
  \to L^2}(1) \quad \text{on }\mathcal{E}_{\mathbb{S}^d}(N_j),$$
where $\operatorname{Op}_{k_j}^{\mathbb{S}^d}$ is a semiclassical quantization procedure on the sphere, and $\overline{\mathcal{M}_E}$, $\overline{a}$ are defined in \eqref{eq:compactmoser}, \eqref{eq:barr}, respectively. Altogether by \eqref{eq:uselemma},
\begin{equation}\label{eq:FINAL}
\langle \operatorname{Op}_{\hbar_j}(a)\Psi_{j},\Psi_{j} \rangle=\langle \operatorname{Op}_{k_j}^{\mathbb{S}^d}(b)\phi_j,\phi_j \rangle +o(1)
\end{equation}
Suppose $\mu$ is a semiclassical measure of $\{\Psi_j\}$. After passing to a subsequence of $j \to \infty$ in \eqref{eq:FINAL},
\begin{equation}\label{eq:FINAL2}
\int_{T^*\mathbb{R}^d}a(x,\xi)d\mu=\int_{T^*\mathbb{S}^d}b(u,\eta)d\nu,
\end{equation}
where $\nu$ is a semiclassical measure of the $1$-eigenfunctions
$\{\phi_j\}$ of $-k_j^2\Delta_{\mathbb{S}^d}$. As a semiclassical
measure of eigenfunctions, $\nu$ is a probability measure
supported on $S^*\mathbb{S}^d$ invariant under the
(co)geodesic flow on $S^*\mathbb{S}^d$. By Remark \ref{rem:invsig}, $(1-u_{d+1})\nu$ assigns zero mass to $S_{\mathsf{NP}}^*\mathbb{S}^d$, so it may be restricted to $S^*(\mathbb{S}_{\neq
    \mathsf{NP}}^d)$ without loss. Then \eqref{eq:FINAL2}
implies $$\mu=(\mathcal{M}_E^{-1})_*((1-u_{d+1})\nu|_{S^*(\mathbb{S}_{\neq
    \mathsf{NP}}^d)}),\qquad \overline{\mu}\coloneqq (i_{\Sigma_E})_*\mu=(\overline{\mathcal{M}_E}^{-1})_*((1-u_{d+1})\nu|_{S^*\mathbb{S}^d}).$$  By Lemma \ref{lem:invsig}, $\mu$
is a probability measure supported on $\Sigma_E$ invariant under
the regularized Moser flow.
\par Conversely, we can essentially follow the proof backwards. Let $E<0$ and suppose $\mu$ is a probability measure supported on $\Sigma_E$ invariant under
the regularized Moser flow, and let $\overline{\mu} \coloneqq
(i_{\Sigma_E})_*\mu$. By Lemma \ref{lem:invsig}, there exists a
unique probability measure $\nu$ on $S^*\mathbb{S}^d$ invariant
under the cogeodesic flow such that
$$(1-u_{d+1})\nu \coloneqq
(\overline{\mathcal{M}_E})_*\overline{\mu}.$$
By \cite{JZ99}, $\nu$ is a semiclassical measure of a
sequence $\{\phi_j\}$ of $1$-eigenfunctions of
$-k_j^2\Delta_{\mathbb{S}^d}$ where $k_j=1/\sqrt{N_j(N_j+d-1)}$. In particular, for all $a \in C_c^{\infty}(T^*\mathbb{R}^d)$, if $b \in C_c^{\infty}(T^*\mathbb{S}^d)$ is such that $b|_{S^*\mathbb{S}^d}=(1-u_{d+1})\overline{\mathcal{M}_E^{-1}}^*\overline{a}$,
\begin{equation}\label{eq:convFINAL}
\langle \operatorname{Op}_{k_j}^{\mathbb{S}^d}(b)\phi_j,\phi_j \rangle \to \int_{S^*\mathbb{S}^d}b\ d\nu=\int_{\overline{\Sigma_E}}\overline{a} d \overline{\mu}=\int_{\Sigma_E}\overline{a}\circ i_{\Sigma_E} d\mu=\int_{\Sigma_E}ad\mu.
\end{equation}
We are done once we define $\Psi_j \coloneqq \mathcal{V}_{\hbar_j,E_{\hbar_j}}^{-1}[\phi_j]$ (where $\hbar_j \coloneqq \frac{1}{\sqrt{-2E}(N_j+\frac{d-1}{2})}$ and $E_{\hbar_j} \coloneqq E$, implying $\hbar_j/k_j \to p_0^{-1}$) and combine
\eqref{eq:convFINAL} with \eqref{eq:FINAL}.
\subsection{Proof of Theorem \ref{theo:2}}
The proof of Theorem \ref{theo:2} is essentially the same as
that of Theorem \ref{theo:1}. Indeed, up to \eqref{eq:FINAL},
the proofs are the same up to changing Lemma \ref{lem:ext1} to
the stronger Lemma \ref{lem:ext2}, utilizing the fact that
$$\frac{2}{1+|x|^2}R_{\pi/2}^*\mathsf{S}_{p_0^{-1}}^*\mathcal{D}_{p_0^{-1}}^*a=R_{\pi/2}^*\mathsf{S}_{p_0^{-1}}^*\mathcal{D}_{p_0^{-1}}^*\frac{2p_0^2}{|\xi|^2+p_0^2}a$$
along the fact that $\frac{2p_0^2}{|\xi|^2+p_0^2}=p_0^2|x|$ on
$\Sigma_E$ and \eqref{eq:timechange3}. Equation
\eqref{eq:FINAL2} then becomes
\begin{equation}\label{eq:FINAL3}
\int_{\overline{\Sigma_E}}\overline{a} d\overline{\mu}=\int_{S^*\mathbb{S}^d}(1-u_{d+1})(\overline{\mathcal{M}_E}^{-1})^*\overline{a
} d\nu, 
\end{equation}
showing
$\overline{\mu}=(\overline{\mathcal{M}_E}^{-1})_*[(1-u_{d+1})\nu]$.

This requires a quick proof since we only showed \eqref{eq:FINAL3} for $\overline{a} \in S_{\overline{\Sigma_E}}$. Setting $\overline{a}\equiv 1$ in \eqref{eq:FINAL3} proves $\overline{\mu}$ is a probability measure since $(\overline{\mathcal{M}_E}^{-1})_*[(1-u_{d+1})\nu]$ is a probability measure. Furthermore, both measures agree when restricted to subsets of $\Sigma_E$ by considering $a \in C_c^{\infty}(T^*\mathbb{R}^d)$. Finally, Remark \ref{rem:invsig} implies $(\overline{\mathcal{M}_E}^{-1})_*[(1-u_{d+1})\nu] (\overline{\Sigma_E}\setminus \Sigma_E)=0$, which implies $\overline{\mu}(\overline{\Sigma_E}\setminus \Sigma_E)=0$ because otherwise would contradict $\overline{\mu}(\overline{\Sigma_E})=1$. This shows $\overline{\mu}=(\overline{\mathcal{M}_E}^{-1})_*[(1-u_{d+1})\nu]$ and hence $\overline{\mu}$
is a probability measure on $\overline{\Sigma_E}$ invariant under
the regularized Moser flow.

The proof of the converse is also relatively unchanged from the original. We modify the proof slightly by simply considering test functions $b \in C_c^{\infty}(T^*\mathbb{S}^d)$ such that $b|_{S^*\mathbb{S}^d}=(1-u_{d+1})(\overline{\mathcal{M}_E}^{-1})^*\overline{a}$, and the same argument goes through.
\section{Extending operators to \texorpdfstring{$\Psi_\hbar^0(\mathbb S^d)$}{Psi h 0(Sd)}}\label{sec:extension}
In this section, we prove extension lemmas used in the proofs of Theorems \ref{theo:1} and \ref{theo:2}. Lemmas \ref{lem:ext1} and \ref{lem:ext2} are used in the proofs of Theorem \ref{theo:1} and \ref{theo:2}, respectively. With regards to the proofs of the theorems, Lemma \ref{lem:ext2} is stronger and can be applied to both proofs, so the reader may skip Lemma \ref{lem:ext1}. We include it since it is more elementary and easier to prove. 
\begin{lem}\label{lem:ext1}
Suppose $ a\in C_c^{\infty}(T^*\mathbb{R}^d)$ and define the
following operator on $C^{\infty}(\mathbb{S}^d)$
$$B_{\hbar}\coloneqq E \circ
\widehat{\omega}\circ \operatorname{Op}_{\hbar}^L(a) \circ \widehat{\omega}^{-1}
\circ R,$$
where $\operatorname{Op}_{\hbar}^L$ is semiclassical left quantization, $\widehat{\omega}$ is defined in \eqref{eq:fockmaps}, $R$ is restriction to
$\mathbb{S}_{\neq \mathsf{NP}}^d$, and $E$ extends
functions on $\mathbb{S}_{\neq \mathsf{NP}}^d$ to functions on
$\mathbb{S}^d$ by $0$ at $\mathsf{NP}$. Then $B_{\hbar}\in
\Psi_{\hbar}^{0}(\mathbb{S}^d)$ with principal
symbol $$ C^{\infty}(T^*\mathbb{S}^d)\ni \sigma_{\hbar}(B_{\hbar})(u,\eta)=\begin{cases}
                                              a(\omega^{-1}(u),(d_u\omega^{-1})^{-T}(\eta))
                                               &\text{if
                                              }(u,\eta) \in
                                                 T^*\mathbb{S}_{\neq
                                                 \mathsf{NP}}^d\\
                                               0 &\text{else}.

                                              \end{cases} $$
                                              \end{lem}
                                              \begin{rem}\label{rem:ext}
                                              We prove this
                                              lemma for left
                                              quantization
                                              $\operatorname{Op}_{\hbar}^L$
                                              because the
      Schwartz kernel is
                                              easier to manipulate, but
                                              the same statement is true
                                              for Weyl
                                              quantization (with
                                              a more involved
                                              proof). One can
                                              avoid thinking
                                              about this
                                              complication by
                                              switching to left
                                              quantization in
                                              the proof of
                                              Theorem
                                              \ref{theo:1} by
                                              $$\operatorname{Op}_{\hbar}^L(a)=\operatorname{Op}_{\hbar}(a)+O_{L^2 \to L^2}(\hbar), $$
which can be found in \cite[Proposition 2.23]{N25}.
                                              \end{rem}
                                            
                                              \begin{proof}
Let $L\gg 0$ be such that if $|x|>L$ then $a(x,\xi)=0$ for all
$\xi$. Let $\zeta \in C_c^{\infty}(\mathbb{R}^d;[0,1])$ be $1$ on
$|x| \leq L+1$ and define $\chi \coloneqq \zeta \circ \omega^{-1}
\in C^{\infty}(\mathbb{S}^d).$ Since $\zeta \operatorname{Op}_\hbar^L(a)=\operatorname{Op}_\hbar^L(a)$,
$$B_{\hbar}=\chi B_{\hbar}=\chi B_{\hbar}\chi+\underbrace{\chi B_{\hbar}(1-\chi)}_{\eqqcolon R_{\hbar}},$$
we observe that $\chi B_{\hbar}\chi \in
\Psi_{\hbar}^0(\mathbb{S}^d)$ by definition. It suffices to show
$R_{\hbar}\in \hbar^{\infty}\Psi_{\hbar}^{-\infty}$. The
Schwartz kernel of $\operatorname{Op}_{\hbar}^L(a)$ is given by
$K_{\hbar}(x,y)=\frac{1}{(2\pi
  \hbar)^d}\int_{\mathbb{R}^d}a(x,\xi)e^{-\frac{i}{\hbar}(x-y)\xi}d\xi$,
so the Schwartz kernel of $R_{\hbar}$ is given by
\begin{equation}\label{eq:remain}
R_{\hbar}(u,v)=\frac{\chi(u)(1-\chi(v))\langle \omega^{-1}(u)\rangle^d \langle \omega^{-1}(v)\rangle^d}{2^d(2\pi \hbar)^d} \int_{\mathbb{R}^d}a(\omega^{-1}(u),\xi)e^{-\frac{i}{\hbar}(\omega^{-1}(u)-\omega^{-1}(v))\xi}d\xi 
\end{equation}
for $u,v \in \mathbb{S}_{\neq \mathsf{NP}}^d$ and
$R_{\hbar}(\mathsf{NP},\bullet)=R_{\hbar}(\bullet,\mathsf{NP})=0$,
where $\langle \bullet \rangle \coloneqq (1+|\bullet|^2)^{1/2}$ and we accounted for the Jacobian factor $(
\frac{|\omega^{-1}(v)|^2+1}{2})^{d}$ coming from the change of
variables from $\mathbb{R}^d$ to $\mathbb{S}_{\neq
  \mathsf{NP}}^d$. 

We first show $R_{\hbar}\in C^{\infty}(\mathbb{S}^d \times
\mathbb{S}^d)$. It is clear that $R_{\hbar}\in C^{\infty}(\mathbb{S}_{\neq \mathsf{NP}}^d \times
\mathbb{S}_{\neq \mathsf{NP}}^d) \cap C^{\infty}(U_{\mathsf{NP}}\times \mathbb{S}^d)$ since $R_{\hbar}(u,v)\equiv 0$
when $u$ is in a sufficiently small neighborhood of
$\mathsf{NP}$ (which we call $U_{\mathsf{NP}}$). On the other hand, to see $R_{\hbar}\in C^{\infty}(\mathbb{S}_{\neq \mathsf{NP}}^d \times U_{\mathsf{NP}})$, we set $u=\omega(x)$  and $v=\omega_{\mathsf{SP}}(y)\coloneqq \omega(-y/|y|^2)$ so that for $y \neq 0$
$$R_{\hbar}=\zeta(x)\langle x \rangle^d\frac{1}{2^d(2\pi \hbar)^d}\Big( \frac{1+|y|^2}{|y|^2}\Big)^{d/2}\int_{\mathbb{R}^d}a(x,\xi)e^{-\frac{i}{\hbar}(x+y/|y|^2)\xi}d\xi.$$
We first observe that on $ \operatorname{supp} R_{\hbar}$, $x$
lies in a compact set, which implies
$|x+\frac{y}{|y|^2}|>\frac{c}{|y|}$ for some $c>0$ and for all
$0\neq y\in\omega_{\mathsf{SP}}^{-1}(U_{\mathsf{NP}})$. Using
repeated integration by parts with the operator $L_{\hbar}
\coloneqq- \frac{\hbar(x+y/|y|^2)}{i|x+y/|y||^2}\partial_{\xi}$, it can be shown that for all multi-indeces $\alpha,\beta$ and $N,M>0$
\begin{equation}\label{eq:Rest}
|\partial_x^{\alpha}\partial_y^{\beta}R_{\hbar}|\leq C_{\alpha,\beta,N,M}\hbar^N|y|^M, \qquad \text{for all } x \in \mathbb{R}^d,\ 0 \neq y \in \omega_{\mathsf{SP}}^{-1}(U_{\mathsf{NP}}).
\end{equation}
Indeed, any potential terms of $1/|y|$ coming from differentiating in $y$ can be counteracted by a sufficient amount of integration by parts since $|x+\frac{y}{|y|^2}|^{-1}\leq \frac{|y|}{c}$. The above estimate proves that $R_{\hbar}$ extends smoothly to $0$ over $y=0$, proving $R_{\hbar}\in C^{\infty}(\mathbb{S}_{\neq \mathsf{NP}}^d \times U_{\mathsf{NP}})$ and thus $R_{\hbar}\in C^{\infty}(\mathbb{S}^d \times
\mathbb{S}^d)$.
\par Now we need to show that all of the
$C^{\infty}$-seminorms of $R_{\hbar}$ are
$O(\hbar^{\infty})$. That is, if $\{(U_j,\phi_j)\}_{j=1}^4$ is
the finite atlas of $\mathbb{S}^d\times \mathbb{S}^d$ given by
Cartesian products of $\mathbb{S}_{\neq \mathsf{NP}}^d$ and
$U_{\mathsf{NP}}$ (along with products of $\omega_{\mathsf{NP}}^{-1}$
and $\omega_{\mathsf{SP}}^{-1}$), we need to show that for a partition of unity $\chi_j\in C_c^{\infty}(U_j;[0,1])$ subordinate to this atlas, we have
\begin{equation}\label{eq:seminorm}
\lVert R_{\hbar}\rVert_k \coloneqq \sum_{j=1}^4\sum_{|\alpha|\leq k}\sup_{x \in \phi_j(U_j)}\big| \partial^{\alpha}\big( (\chi_jR_{\hbar})\circ \phi_j^{-1}\big)(x)| = O(\hbar^{\infty}).
\end{equation}
The terms in the sum involving the two charts of the form $U_{\mathsf{NP}}\times \bullet$ are trivially $0$, and the term with the chart $\mathbb{S}_{\neq \mathsf{NP}}^d \times U_{\mathsf{NP}}$ can be shown to be $O(\hbar^{\infty})$ by \eqref{eq:Rest}. The term on the last chart $\mathbb{S}_{\neq \mathsf{NP}}^d \times \mathbb{S}_{\neq \mathsf{NP}}^d$ is $O(\hbar^{\infty})$ by a similar integration by parts argument on \eqref{eq:remain} where $u=\omega(x),v=\omega(y)$, and we are done.
\end{proof}

The main result of this section is the following lemma.
\begin{lem}\label{lem:ext2}
Suppose $a \in S(1)$ and $a \circ
(\omega^*)^{-1}|_{S^*\mathbb{S}_{\neq \mathsf{NP}}^d}$ extends to a smooth function $a_{\omega}$ on $S^{*}\mathbb{S}^d$, where
$\omega^*$ is defined in \eqref{eq:mos}. Then there exists
$B_{\hbar}\in \Psi_{\hbar}^0(\mathbb{S}^d)$ satisfying
\begin{equation}\label{eq:}
(B_{\hbar}-\widehat{\omega} \circ \operatorname{Op}_{\hbar}(a) \circ \widehat{\omega}^{-1})\Pi_{\hbar}=o_{L^2\to L^2}(1),
\end{equation}
as $\hbar \to 0$, where $\Pi_{\hbar}:L^2(\mathbb{S}^d) \to
L^2(\mathbb{S}^d)$ is orthogonal projection onto the
1-eigenspace of $-\hbar^2\Delta_{\mathbb{S}^d}$ and
$\widehat{\omega}$ is defined in
\eqref{eq:fockmaps}. Furthermore, $b \coloneqq
\sigma_{\hbar}(B_{\hbar}) \in C_c^{\infty}(T^*\mathbb{S}^d)$
satisfies
\begin{equation}\label{eq:princb}
b|_{S^*\mathbb{S}^d}=a_{\omega}.
\end{equation}
\end{lem}
The idea of the proof is to use that $\widehat{\omega}
\circ \operatorname{Op}_{\hbar}(a) \circ \widehat{\omega}^{-1}$
is (morally) a pseudodifferential operator away from $\mathsf{NP}$, microlocally on
$S^*\mathbb{S}^d$, and we use an
eigenfunction concentration estimate to handle the $L^2$ mass at
$\mathsf{NP}$. It is worth noting that $\widehat{\omega} \circ
\operatorname{Op}_{\hbar}(a) \circ \widehat{\omega}^{-1}$ is
\textit{not} generally in $\Psi_{\hbar}^{0}(\mathbb{S}_{\neq
  \mathsf{NP}}^d)$ since the class $S(1)$ is not coordinate
invariant, but microlocalizing to the cosphere bundle
$S^*\mathbb{S}^d$ solves this issue, as proved below. 
\begin{proof}
Let $\theta \in C_c^{\infty}((-1,1);[0,1])$ be $1$ near the origin, and define $ b\in C_c^{\infty}(T^*\mathbb{S}^d)$ by
\begin{equation}\label{eq:BB}
b(u,\eta) \coloneqq \theta(|\eta|_g^2-1)a_{\omega}\Big(u,\frac{\eta}{|\eta|_{g}}\Big), \quad B_{\hbar} \coloneqq \operatorname{Op}_{\hbar}^{\mathbb{S}^d}(b),
\end{equation}
where $|\bullet|_g$ denotes the norm on covectors induced by the
round metric on $\mathbb{S}^d$, and
$\operatorname{Op}_{\hbar}^{\mathbb{S}^d}$ is a semiclassical
quantization procedure on the sphere. Define
\begin{equation}\label{eq:dd}
Q_{\hbar}\coloneqq B_{\hbar}-\widehat{\omega}  \circ \operatorname{Op}_{\hbar}(a) \circ \widehat{\omega}^{-1}.
\end{equation}
We want to show $Q_{\hbar}\Pi_{\hbar}=o_{L^2 \to L^2}(1)$. We
analyze this at and away from $\mathsf{NP}$ separately. Let
$\delta>0$ and $\varphi_{\delta} \in C^{\infty}(\mathbb{S}^d)$
be $1$ near
$\mathsf{NP}$ and supported in $\operatorname{dist}_{\mathbb{S}^d}(\mathsf{NP},\bullet)<\delta$. We split
\begin{equation}\label{eq:bananasplit}
Q_{\hbar}\Pi_{\hbar}=\varphi_{\delta}Q_{\hbar}\Pi_{\hbar}+(1-\varphi_{\delta})Q_{\hbar}\Pi_{\hbar}.
\end{equation}
\underline{\emph{Step 1: Estimate away from $\mathsf{NP}$: }} From
\eqref{eq:bananasplit}, we estimate
$(1-\varphi_{\delta})Q_{\hbar}\Pi_{\hbar}$ by showing
\begin{equation}\label{eq:step1main}
\big\lVert (1-\varphi_{\delta})Q_{\hbar}\Pi_{\hbar} \big\rVert_{L^2 \to L^2}=O_{\delta}(\hbar)
\end{equation}
Let
$\varphi_{\delta}' \in C^{\infty}(\mathbb{S}^d)$ be supported where $\varphi_{\delta}=1$ and $1$ near $\mathsf{NP}$. We write
\begin{equation}\label{eq:split2}
(1-\varphi_{\delta})Q_{\hbar}\Pi_{\hbar}=\underbrace{(1-\varphi_{\delta})Q_{\hbar}\varphi_{\delta}'\Pi_{\hbar}}_{\eqqcolon (\mathrm{I})}+\underbrace{(1-\varphi_{\delta})Q_{\hbar}(1-\varphi_{\delta}')\Pi_{\hbar}}_{\eqqcolon \mathrm{(II)}}.
\end{equation}
For (I), since the supports of $1-\varphi_{\delta}$ and
$\varphi_{\delta}'$ are disjoint, it can be shown that
$$\big\lVert(1-\varphi_{\delta})Q_{\hbar}\varphi_{\delta}'\big\rVert_{L^2
  \to L^2}=O_{\delta}(\hbar^{\infty}).$$
Indeed, $Q_{\hbar}$ is comprised of $B_{\hbar}$ and $\widehat{\omega} \circ \operatorname{Op}_{\hbar}(a) \circ \widehat{\omega}^{-1}$
by \eqref{eq:dd}. For $B_{\hbar}$, this follows from the usual
pseudo-local property of $\Psi_{\hbar}^0(\mathbb{S}^d)$
\cite[(14.2.6)]{Z12}, and for $\widehat{\omega}  \circ \operatorname{Op}_{\hbar}(a) \circ \widehat{\omega}^{-1}$, after conjugating by
$\widehat{\omega}$, this follows from a similar statement in the
$S(1)$ calculus on $\mathbb{R}^d$ \cite[Theorem 4.25]{Z12}.
\par The operator (II) requires more care. Because $-\hbar^2
\Delta_{\mathbb{S}^d}\Pi_{\hbar}=\Pi_{\hbar}$, functional
calculus gives
$\chi(-\hbar^2\Delta_{\mathbb{S}^d})\Pi_{\hbar}=\Pi_{\hbar}$,
where $\chi \in C_c^{\infty}(\mathbb{R};[0,1])$ is supported
near $1$ such that $\chi=1$ near $1$. It suffices to show
\begin{equation}\label{eq:suffice}
\big\lVert
(1-\varphi_{\delta})Q_{\hbar}(1-\varphi_{\delta}')\chi(-\hbar^2\Delta_{\mathbb{S}^d})\Pi_{\hbar}\big\rVert_{L^2
  \to L^2}=O_{\delta}(\hbar).
\end{equation}
To this end, we recall $\chi(-\hbar^2\Delta_{\mathbb{S}^d}) \in
\Psi_{\hbar}^{\operatorname{comp}}(\mathbb{S}^d)$ with
$\operatorname{WF}_{\hbar}(\chi(-\hbar^2\Delta_{\mathbb{S}^d}))
\subset (|\bullet|_g^2)^{-1}(\operatorname{supp}\chi)$ (see \cite[Theorem 14.9]{Z12}), and since
$1-\varphi_{\delta}'$ is supported away from the north pole,
\begin{equation}\label{eq:weirdcompact}
(1-\varphi_{\delta}')\chi(-\hbar^2\Delta_{\mathbb{S}^d}) \in \Psi_{\hbar}^{\mathrm{comp}}(\mathbb{S}_{\neq \mathsf{NP}}^d),
\end{equation}
where $\Psi_{\hbar}^{\mathrm{comp}}$ denotes compactly microlocalized semiclassical pseudodifferential operators (see \cite[Definition E.28]{DZ19}).
Then $$T_{\hbar}\coloneqq \widehat{\omega}^{-1} \circ
(1-\varphi_{\delta}')\chi(-\hbar^2\Delta_{\mathbb{S}^d}) \circ
\widehat{\omega} \in \Psi_{\hbar}^{\mathrm{comp}}(\mathbb{R}^d).$$ Define
$\zeta \in C_c^{\infty}(T^*\mathbb{R}^d)$ to be $1$ near
$\operatorname{WF}_{\hbar}(T_{\hbar})$. Then
$\operatorname{Op}_{\hbar}(\zeta)T_{\hbar}=T_{\hbar}+O_{\delta}(\hbar^{\infty})_{L^2
  \to L^2}$, so for \eqref{eq:suffice}, it suffices to show
\begin{equation}\label{eq:suffice2}
\big\lVert
(1-\varphi_{\delta})\tilde{Q}_{\hbar}(1-\varphi_{\delta}')\chi(-\hbar^2\Delta_{\mathbb{S}^d})\Pi_{\hbar}\big\rVert_{L^2
  \to L^2}=O_{\delta}(\hbar),
\end{equation}
where $\tilde{Q}_{\hbar} \coloneqq
B_{\hbar}-\widehat{\omega}\circ
\operatorname{Op}_{\hbar}(a\zeta) \circ
\widehat{\omega}^{-1}$. The upshot of this is that
$$\tilde{B}_{\hbar} \coloneqq (1-\varphi_{\delta})\tilde{Q}_{\hbar}(1-\varphi_{\delta}')\chi(-\hbar^2\Delta_{\mathbb{S}^d})
\in \Psi_{\hbar}^{\operatorname{comp}}(\mathbb{S}^d), $$
with principal symbol
$$\sigma_{\hbar}(\tilde{B}_{\hbar})(u,\eta)= \big(1-\varphi_{\delta}(u)\big)\big(b(u,\eta)-a\circ
(\omega^*)^{-1}(u,\eta)\big)\chi(|\eta|_g^2) \in C_c^{\infty}(T^*\mathbb{S}^d).$$
This symbol
vanishes on $S^*\mathbb{S}^d$ by construction, and since $d (|\eta|_{g}^2-1)\neq
0$ near $S^*\mathbb{S}^d$, we see
\begin{equation}\label{eq:almost2}
(1-\varphi_{\delta})\tilde{Q}_{\hbar}(1-\varphi_{\delta}')\chi(-\hbar^2\Delta_{\mathbb{S}^d})=\operatorname{Op}_{\hbar}^{\mathbb{S}^d}(c)(-\hbar^2\Delta_{\mathbb{S}^d}-1)+O_{\delta}(\hbar)_{L^2
  \to L^2},
\end{equation}
where $c(u,\eta) \coloneqq
\sigma_{\hbar}(\tilde{B}_{\hbar})(u,\eta)/(|\eta|_g^2-1) \in
C_c^{\infty}(T^*\mathbb{S}^d)$. By applying $\Pi_{\hbar}$ to
\eqref{eq:almost2}, we obtain \eqref{eq:suffice2} and hence
\eqref{eq:suffice}.\\
\underline{\emph{Step 2: Estimate near $\mathsf{NP}$: }} From
\eqref{eq:bananasplit}, we estimate
$\varphi_{\delta}Q_{\hbar}\Pi_{\hbar}$ by showing
\begin{equation}\label{eq:step2main}
\big\lVert \varphi_{\delta} Q_{\hbar}\Pi_{\hbar} \big\rVert_{L^2 \to L^2}=O(\delta^{1/2})+O_{\delta}(\hbar^{\infty}),
\end{equation}
where $O(\delta^{1/2})$ is independent of $\hbar$ for $\hbar<\delta$. Let
$\psi_{\delta} \in C^{\infty}(\mathbb{S}^d)$ be 1 near $\operatorname{supp} \varphi_{\delta}$ and supported in $\operatorname{dist}_{\mathbb{S}^d}(\bullet,\mathsf{NP})<2\delta$. We write
\begin{equation}\label{eq:splitstep2}
\varphi_{\delta}Q_{\hbar}\Pi_{\hbar}=\underbrace{\varphi_{\delta}Q_{\hbar}(1-\psi_{\delta})\Pi_{\hbar}}_{\eqqcolon (\mathrm{III})}+\underbrace{\varphi_{\delta}Q_{\hbar}\psi_{\delta}\Pi_{\hbar}}_{\eqqcolon \mathrm{(IV)}}.
\end{equation}
The operator (III) can be shown to be in
$O_{\delta}(\hbar^{\infty})_{L^2 \to L^2}$ by the same argument
as (I) in the previous step. \par For (IV), we first note that
$\varphi_{\delta}Q_{\hbar}$ is $L^2 \to L^2$ bounded independent
of $\hbar,\delta$ using the easy bound $\varphi_{\delta}\leq
1$ and the fact that $B_{\hbar},\widehat{\omega} \circ \operatorname{Op}_{\hbar}(a) \circ \widehat{\omega}^{-1}$ are both $L^2 \to L^2$ bounded independent
of $\hbar$ by Calder\'on-Vaillancourt. For $\psi_{\delta}\Pi_{\hbar}$, we use the ``trivial''
$L^2$-concentration bound of eigenfunctions
\begin{equation}\label{eq:presogge}
\lVert \mathbf{1}_{\operatorname{dist}_{\mathbb{S}^d}(\bullet,\mathsf{NP})<2\delta}\Pi_{\hbar}\rVert_{L^2\to L^2} =O(\delta^{1/2}),\quad \hbar < \delta < \pi/2,
\end{equation}
where the implied constant is independent of
$\hbar$ (see \cite[(4.1)]{S16}). This implies
\begin{equation}\label{eq:sogge}
\lVert \psi_{\delta}\Pi_{\hbar}\rVert_{L^2\to L^2} =O(\delta^{1/2}),\quad \hbar < \delta < \pi/2,
\end{equation}
Overall, for (IV) we then have
\begin{equation}\label{eq:(IV)}
\lVert \varphi_{\delta}Q_{\hbar}\psi_{\delta}\Pi_{\hbar}\rVert_{L^2 \to L^2}=O(\delta^{1/2}),
\end{equation}
proving \eqref{eq:step2main}.\\
\underline{\emph{Step 3: Combining the estimates: }}
To finish off the proof, \eqref{eq:step1main} and
\eqref{eq:step2main} together give
$$\lVert Q_{\hbar}\Pi_{\hbar}\rVert_{L^2 \to
  L^2}\leq O(\delta^{1/2})+O_{\delta}(\hbar)$$
for any $\delta\in (\hbar,\pi/2)$. Taking $\limsup_{\hbar \to 0}$ first and then $\delta \to 0$ proves $\lVert Q_{\hbar}\Pi_{\hbar}\rVert_{L^2 \to
  L^2}=o(1)$ as $\hbar \to 0$, as desired.
\end{proof}

\bibliographystyle{alpha-reverse}
\bibliography{ultimaterefs}

\begin{thebibliography}{Loh25b}

\bibitem[AM22]{AM22}
Arnaiz, V. and Maci\`a, F.
\newblock Localization and delocalization of eigenmodes of harmonic
  oscillators.
\newblock {\em Proc. Amer. Math. Soc.}, 150(5):2195--2208, 2022.

\bibitem[Arn89]{A89}
Arnol'd, V.~I.
\newblock {\em Mathematical methods of classical mechanics}, volume~60 of {\em
  Graduate Texts in Mathematics}.
\newblock Springer-Verlag, New York, second edition, 1989.
\newblock Translated from the Russian by K. Vogtmann and A. Weinstein.

\bibitem[Arn20]{A20}
Arnaiz, V.
\newblock Spectral stability and semiclassical measures for renormalized {KAM}
  systems.
\newblock {\em Nonlinearity}, 33(6):2562--2591, 2020.

\bibitem[BI66]{BI66}
Bander, M. and Itzykson, C.
\newblock Group theory and the hydrogen atom. {I}, {II}.
\newblock {\em Rev. Modern Phys.}, 38:330--345; 346--358, 1966.

\bibitem[CdV85]{CdV85}
Colin~de Verdi\`ere, Y.
\newblock Ergodicit\'e{} et fonctions propres du laplacien.
\newblock {\em Comm. Math. Phys.}, 102(3):497--502, 1985.

\bibitem[CJK08]{CJK08}
Castella, F.~c., Jecko, T., and Knauf, A.
\newblock Semiclassical resolvent estimates for {S}chr\"odinger operators with
  {C}oulomb singularities.
\newblock {\em Ann. Henri Poincar\'e}, 9(4):775--815, 2008.

\bibitem[DM78]{DM78}
Dellacherie, C. and Meyer, P.-A.
\newblock {\em Probabilities and potential}, volume~29 of {\em North-Holland
  Mathematics Studies}.
\newblock North-Holland Publishing Co., Amsterdam-New York, 1978.

\bibitem[Dya22]{D22}
Dyatlov, S.
\newblock Around quantum ergodicity.
\newblock {\em Ann. Math. Qu\'e.}, 46(1):11--26, 2022.

\bibitem[DZ19]{DZ19}
Dyatlov, S. and Zworski, M.
\newblock {\em Mathematical theory of scattering resonances}, volume 200 of
  {\em Graduate Studies in Mathematics}.
\newblock American Mathematical Society, Providence, RI, 2019.

\bibitem[Eul67]{E74}
Euler, L.
\newblock De motu rectilineo trium corporum se mutuo attrahentium.
\newblock {\em Novi Comm. Acad. Sci. Petrop.}, 11:144--151, 1767.

\bibitem[Foc35]{F35}
Fock, V.
\newblock {Z}ur {T}heorie des {W}asserstoffatoms.
\newblock {\em Zeitschrift f\"{u}r Physik}, 98:145--154, 1935.

\bibitem[GK91]{GK91}
G{\'e}rard, C. and Knauf, A.
\newblock Collisions for the quantum {C}oulomb {H}amiltonian.
\newblock {\em Comm. Math. Phys.}, 143(1):17--26, 1991.

\bibitem[Gol75]{G75}
Goldstein, H.
\newblock {Prehistory of the ``Runge–Lenz" vector}.
\newblock {\em American Journal of Physics}, 43(8):737--738, 1975.

\bibitem[Gol76]{G76}
Goldstein, H.
\newblock {More on the prehistory of the Laplace or Runge–Lenz vector}.
\newblock {\em American Journal of Physics}, 44(11):1123--1124, 1976.

\bibitem[GS90]{GS90}
Guillemin, V. and Sternberg, S.
\newblock {\em Variations on a theme by {K}epler}, volume~42 of {\em American
  Mathematical Society Colloquium Publications}.
\newblock American Mathematical Society, Providence, RI, 1990.

\bibitem[GW24]{GW24}
Galkowski, J. and Wunsch, J.
\newblock Propagation for {S}chr\"{o}dinger operators with potentials singular
  along a hypersurface.
\newblock {\em Arch. Ration. Mech. Anal.}, 248(3):Paper No. 37, 28, 2024.

\bibitem[Ham47]{H47}
Hamilton, W.~R.
\newblock The hodograph or a new method of expressing in symbolic language the
  {N}ewtonian law of attraction.
\newblock {\em Proc. Royal Irish Acad.}, 3(19):344--353, 1847.

\bibitem[HM23]{HM23}
Hillairet, L. and Marzuola, J.~L.
\newblock Eigenvalue spacing for 1{D} singular {S}chr\"{o}dinger operators.
\newblock {\em Asymptot. Anal.}, 133(1-2):267--289, 2023.

\bibitem[HS96]{HS96}
Hislop, P.~D. and Sigal, I.~M.
\newblock {\em Introduction to spectral theory}, volume 113 of {\em Applied
  Mathematical Sciences}.
\newblock Springer-Verlag, New York, 1996.
\newblock With applications to Schr\"{o}dinger operators.

\bibitem[JZ99]{JZ99}
Jakobson, D. and Zelditch, S.
\newblock Classical limits of eigenfunctions for some completely integrable
  systems.
\newblock In {\em Emerging applications of number theory ({M}inneapolis, {MN},
  1996)}, volume 109 of {\em IMA Vol. Math. Appl.}, 329--354. Springer, New
  York, 1999.

\bibitem[Kat95]{K95}
Kato, T.
\newblock {\em Perturbation theory for linear operators}.
\newblock Classics in Mathematics. Springer-Verlag, Berlin, 1995.
\newblock Reprint of the 1980 edition.

\bibitem[Ker05]{K05}
Keraani, S.
\newblock Wigner measures dynamics in a {C}oulomb potential.
\newblock {\em J. Math. Phys.}, 46(6):063512, 21, 2005.

\bibitem[KS65]{KS65}
Kustaanheimo, P. and Stiefel, E.
\newblock Perturbation theory of {K}epler motion based on spinor
  regularization.
\newblock {\em J. Reine Angew. Math.}, 218:204--219, 1965.

\bibitem[Kus64]{K64}
Kustaanheimo, P.
\newblock Spinor regularization of the {K}epler motion.
\newblock {\em Ann. Univ. Turku. Ser. A I}, 73:7, 1964.

\bibitem[Laz93]{L93}
Lazutkin, V.~F.
\newblock {\em K{AM} theory and semiclassical approximations to
  eigenfunctions}, volume~24 of {\em Ergebnisse der Mathematik und ihrer
  Grenzgebiete (3) [Results in Mathematics and Related Areas (3)]}.
\newblock Springer-Verlag, Berlin, 1993.
\newblock With an addendum by A. I. Shnirelman.

\bibitem[LC20]{LC20}
Levi-Civita, T.
\newblock Sur la r\'egularisation du probl\`eme des trois corps.
\newblock {\em Acta Math.}, 42(1):99--144, 1920.

\bibitem[LN25]{LN25}
L{\'{e}}autaud, M. and Nonnenmacher, S.
\newblock Introduction to {S}pectral {T}heory.
\newblock
  \url{https://leautaud.perso.math.cnrs.fr/files/spectral-theory-2025.pdf},
  2025.

\bibitem[Loh25a]{L23}
Lohr, N.
\newblock {S}emiclassical measures of eigenfunctions of the attractive
  {C}oulomb operator.
\newblock {\em Ann. Henri Poincar\'e}, 2025.
\newblock Published online.

\bibitem[Loh25b]{L25}
Lohr, N.~G.
\newblock {\em Semiclassical {P}hase {S}pace {D}istributions and {S}cattering
  {T}heory of the {H}armonic {O}scillator and {H}ydrogen {A}tom}.
\newblock ProQuest LLC, Ann Arbor, MI, 2025.
\newblock Thesis (Ph.D.)--Northwestern University.

\bibitem[Mil83]{M83}
Milnor, J.
\newblock On the geometry of the {K}epler problem.
\newblock {\em Amer. Math. Monthly}, 90(6):353--365, 1983.

\bibitem[Mos70]{M70}
Moser, J.
\newblock Regularization of {K}epler's problem and the averaging method on a
  manifold.
\newblock {\em Comm. Pure Appl. Math.}, 23:609--636, 1970.

\bibitem[Non25]{N25}
Nonnenmacher, S.
\newblock An {I}ntroduction to {S}emiclassical {A}nalysis.
\newblock
  \url{https://www.imo.universite-paris-saclay.fr/~stephane.nonnenmacher/enseign/Course-Semiclassical-Analysis2025-total.pdf},
  2025.

\bibitem[RC21]{RC21}
Robert, D. and Combescure, M.
\newblock {\em Coherent states and applications in mathematical physics}.
\newblock Theoretical and Mathematical Physics. Springer, Cham, [2021]
  \copyright 2021.
\newblock Second edition [of 2952171].

\bibitem[RS78]{RS4}
Reed, M. and Simon, B.
\newblock {\em Methods of modern mathematical physics. {IV}. {A}nalysis of
  operators}.
\newblock Academic Press [Harcourt Brace Jovanovich, Publishers], New
  York-London, 1978.

\bibitem[Shn74a]{S74a}
Shnirelman, A.~I.
\newblock Ergodic properties of eigenfunctions.
\newblock {\em Uspehi Mat. Nauk}, 29(6(180)):181--182, 1974.

\bibitem[Shn74b]{S74b}
Shnirelman, A.~I.
\newblock Statistical properties of eigenfunctions.
\newblock In {\em Proceedings of the All-USSR School in Differential Equations
  with Infinite Number of Independent Variables and in Dynamical Systems with
  Infinitely Many Degrees of Freedom ({D}ilijan, {A}rmenia, May 21–June 3,
  1973)}. Armenian Academy of Sciences,, 1974.

\bibitem[Sog16]{S16}
Sogge, C.~D.
\newblock Localized {$L^p$}-estimates of eigenfunctions: a note on an article
  of {H}ezari and {R}ivi\`ere.
\newblock {\em Adv. Math.}, 289:384--396, 2016.

\bibitem[SS71]{SS71}
Stiefel, E.~L. and Scheifele, G.
\newblock {\em Linear and regular celestial mechanics. {P}erturbed two-body
  motion, numerical methods, canonical theory}, volume Band 174 of {\em Die
  Grundlehren der mathematischen Wissenschaften}.
\newblock Springer-Verlag, New York-Heidelberg, 1971.

\bibitem[Tay11]{T11}
Taylor, M.~E.
\newblock {\em Partial differential equations {I}. {B}asic theory}, volume 115
  of {\em Applied Mathematical Sciences}.
\newblock Springer, New York, second edition, 2011.

\bibitem[Tay22]{T22}
Taylor, M.~E.
\newblock Remarks on the hydrogen atom {S}chr\"{o}dinger operator.
\newblock
  \url{https://mtaylor.web.unc.edu/wp-content/uploads/sites/16915/2022/01/hydro.pdf},
  2022.

\bibitem[Ver26]{V26}
Verdasco, S.
\newblock Quantum limits of the laplacian perturbed along a geodesic on
  $\mathbb{S}^{2}$.
\newblock \url{https://arxiv.org/abs/2606.10847}, 2026.

\bibitem[Zel87]{Z87}
Zelditch, S.
\newblock Uniform distribution of eigenfunctions on compact hyperbolic
  surfaces.
\newblock {\em Duke Math. J.}, 55(4):919--941, 1987.

\bibitem[Zwo12]{Z12}
Zworski, M.
\newblock {\em Semiclassical analysis}, volume 138 of {\em Graduate Studies in
  Mathematics}.
\newblock American Mathematical Society, Providence, RI, 2012.

\end{thebibliography}

\end{document}